\documentclass[11pt,oneside,reqno]{article}
\usepackage[margin=1.2in]{geometry}

\usepackage{amsmath,amsthm,amssymb,color}
\usepackage[authoryear]{natbib}
\usepackage{booktabs}

\RequirePackage{amsthm,amsmath,amsfonts,amssymb,mathrsfs,dsfont,mathtools,thmtools}
\RequirePackage{natbib}
\RequirePackage[colorlinks,citecolor=blue,urlcolor=blue]{hyperref}
\usepackage[capitalize]{cleveref}
\RequirePackage{graphicx}

\crefname{equation}{}{}
\crefformat{equation}{\textup{(#2#1#3)}}
\crefrangeformat{equation}{\textup{(#3#1#4)--(#5#2#6)}}
\crefdefaultlabelformat{#2\textup{#1}#3}
\crefname{lemma}{Lemma}{Lemmas}
\crefname{page}{p.}{pp.}

\usepackage{bbm}
\usepackage{mathrsfs}
\usepackage{graphicx}
\usepackage{tikz}
\usepackage{enumitem}
\usetikzlibrary{arrows, automata}
\usetikzlibrary{calc}
\usetikzlibrary{positioning}

\numberwithin{equation}{section}
\allowdisplaybreaks[4]

\theoremstyle{plain}
\newtheorem{theorem}{Theorem}[section]

\newtheorem{lemma}{Lemma}[section]
\newtheorem{corollary}{Corollary}[section]

\theoremstyle{definition}
\newtheorem{example}{Example}[section]
\newtheorem{remark}{Remark}[section]

\makeatletter

\newcount\minute
\newcount\hour
\newcount\hourMins
\def\now{%
\minute=\time%
\hour=\time \divide \hour by 60%
\hourMins=\hour \multiply\hourMins by 60%
\advance\minute by -\hourMins%
\zeroPadTwo{\the\hour}:\zeroPadTwo{\the\minute}%
}
\def\zeroPadTwo#1{\ifnum #1<10 0\fi#1}

\renewcommand{\cite}{\citet}

\def\^#1{\ifmmode {\mathaccent"705E #1} \else {\accent94 #1} \fi}
\def\~#1{\ifmmode {\mathaccent"707E #1} \else {\accent"7E #1} \fi}

\def\*#1{#1^\ast}
\edef\-#1{\noexpand\ifmmode {\noexpand\bar{#1}} \noexpand\else \-#1\noexpand\fi}
\def\>#1{\vec{#1}}
\def\.#1{\dot{#1}}

\def\atop{\@@atop}
\def\*#1{\mathscr{#1}}

\renewcommand{\leq}{\leqslant}
\renewcommand{\geq}{\geqslant}

\newcommand{\eq}{\eqref}

\def\t#1{\tilde{#1}}

\newcommand{\IE}{\mathbbm{E}}
\newcommand{\IP}{\mathbbm{P}}
\newcommand{\Var}{\mathop{\mathrm{Var}}\nolimits}

\newcommand{\IR}{\mathbb{R}}

\def\be#1{\begin{equation*}#1\end{equation*}}
\def\ben#1{\begin{equation}#1\end{equation}}
\def\bes#1{\begin{equation*}\begin{split}#1\end{split}\end{equation*}}
\def\besn#1{\begin{equation}\begin{split}#1\end{split}\end{equation}}

\def\ba#1{\begin{align*}#1\end{align*}}
\def\ban#1{\begin{align}#1\end{align}}

\def\mid{\vert}

\def\beqn#1\eeqn{\begin{align}#1\end{align}}
\def\beq#1\eeq{\begin{align*}#1\end{align*}}

\usepackage{graphicx}
\usepackage{latexsym}
\usepackage{amsmath,amsthm,amssymb,amscd}
\usepackage{epsf,amsmath}

\def\E{{\IE}}
\def\P{{\IP}}

\newcommand{\mcl}[1]{\mathcal{#1}}

\def\blfootnote{\xdef\@thefnmark{}\@footnotetext}

\makeatother

\begin{document}

\title{
Normal approximation for exponential random graphs
}
\author{Xiao Fang$^{1,a}$, Song-Hao Liu$^{2,b}$, Qi-Man Shao$^{3,c}$, Yi-Kun
Zhao$^{1,d}$}
\date{\small \it $^1$Department of Statistics,The Chinese University of Hong
  Kong, $^a$xfang@sta.cuhk.edu.hk, $^d$Zyk3355601STAT@link.cuhk.edu.hk\\
$^2$Department of Statistics and Data Science, Southern University of Science and Technology, $^b$liush@sustech.edu.cn\\
$^3$Department of Statistics and Data Science, SICM, NCAMS, Southern University of Science and Technology, $^c$shaoqm@sustech.edu.cn} 
\maketitle

\noindent{\bf Abstract:} 
The question of whether the central limit theorem (CLT) holds for the total number of edges in exponential random graph models (ERGMs) in the subcritical region of parameters has remained an open problem. In this paper, we establish the CLT. As a result of our proof, we also derive a convergence rate for the CLT, an explicit formula for the asymptotic variance, and the CLT for general subgraph counts. To establish our main result, we develop Stein's method for the normal approximation of general functionals of nonlinear exponential families of random variables, which is of independent interest. In addition to ERGMs, our general theorem can also be applied to other models. A key ingredient needed in our proof for the ERGM is a higher-order concentration inequality, which was known in a subset of the subcritical region called Dobrushin's uniqueness region. We use Stein's method to partially generalize such inequalities to the subcritical region.


\medskip



\section{Introduction}
Exponential random graph models (ERGMs) are frequently used as parametric statistical models in network analysis, especially in the sociology community.
They were suggested for directed networks by \cite{holland1981exponential} and for undirected networks by \cite{frank1986markov}.
A general development of the models is presented in \cite{wasserman1994social}.
We focus on undirected networks and refer to
\cite{bhamidi2011mixing} and \cite{chatterjee2013estimating} for the following formulation of the model.

Let $\mathcal{G}_n$ be the space of all simple graphs\footnote{In this paper, simple graphs mean undirected graphs without self-loops or multiple edges.} on $n$ labeled vertices. 
Let $k\geq 1$ be a positive integer. Let $\beta=(\beta_1,\dots, \beta_k)$ be a vector of real parameters, and let $H_1,\dots, H_k$ be (typically small) simple graphs without isolated vertices. 
For any graph $G\in \mcl{G}_n$ and each graph $H_i$, let $|\text{Hom}(H_i,G)|$ denote the number of homomorphisms of $H_i$ into $G$.
A homomorphism is defined as an injective mapping from the vertex set $\mathcal{V}(H_i)$ of $H_i$ to the vertex set $\mathcal{V}(G)$ of $G$, such that each edge in $H_i$ is mapped to an edge in $G$.
For instance, if $H_i$ is an edge, then $|\text{Hom}(H_i,G)|$ is equal to twice the number of edges in $G$. Similarly, if $H_i$ is a triangle, then $|\text{Hom}(H_i,G)|$ is equal to six times the number of triangles in $G$. Given $\beta=(\beta_1,\dots, \beta_k)$ and $H_1,\dots, H_k$, ERGM assigns probability
\begin{equation}\label{ERGM}
p_{\beta}(G)=\frac{1}{Z(\beta)}\exp\left\{ n^2 \sum_{i=1}^k \beta_i t(H_i, G)\right\}
\end{equation}
to each $G\in \mathcal{G}_n$, where 
\be{
t(H_i, G):=\frac{|\text{Hom}(H_i,G)|}{n^{|\mathcal{V}(H_i)|}}
}
denotes the homomorphism density, 
$|\cdot|$ denotes cardinality when applied to a set, and $Z(\beta)$ is a normalizing constant. The $n^2$ and $n^{|\mathcal{V}(H_i)|}$ factors in \eqref{ERGM} ensure a nontrivial large $n$ limit. In this paper, we always take $H_1$ to be an edge by convention ($H_2,\dots, H_k$ are graphs with at least two edges) and assume $\beta_2,\dots, \beta_k$ are positive ($\beta_1$ can be negative).
Note that if $k=1$, then \eqref{ERGM} is the Erd\H{o}s--R\'enyi model $G(n,p)$ where every edge is present with probability $p=p(\beta)=e^{2\beta_1}(1+e^{2\beta_1})^{-1}$, independent of each other. If $k\geq 2$, \eqref{ERGM} ``encourages" the presence of the corresponding subgraphs. See \cref{rem:1} for a discussion on possibly negative values of $\beta_2,\dots, \beta_k$.

Because of the nonlinear nature of \eqref{ERGM}, ERGMs are notoriously more difficult to analyze than classical exponential families of distributions.
To introduce the groundbreaking works by \cite{bhamidi2011mixing} and \cite{chatterjee2013estimating}, we define
\begin{equation}\label{varph1}
\Phi_\beta(a):=\sum_{i=1}^k \beta_i e_i a^{e_i-1},\quad \varphi_\beta(a):=\frac{e^{2\Phi_\beta(a)}}{e^{2\Phi_\beta(a)}+1},
\end{equation}
where $e_i$ is the number of edges in the graph $H_i$. The so-called subcritical region (cf. \cite{bhamidi2011mixing} and \cite{chatterjee2013estimating}) contains all the parameters $\beta=(\beta_1,\dots, \beta_k)$ such that
\begin{equation}
\label{subcritical}
    \text{there is a unique solution $p:=p(\beta)$ to the equation $\varphi_\beta(a)=a$ in $(0,1)$ and $\varphi'(p)<1$.}
\end{equation}
 We always use $p$ to denote the unique solution in the rest of the paper. It can be verified that $p$ satisfies
\begin{equation}\label{eq:p}
2\Phi_\beta(p)=\log(\frac{p}{1-p}).
\end{equation}
\cite{bhamidi2011mixing} and \cite{chatterjee2013estimating} proved that
in the subcritical region, 
ERGM behaves similarly to an Erd\H{o}s--R\'enyi random graph $G(n, p)$ in terms of the asymptotic independence of edges and large deviations within the space of graphons, respectively.

Recently, \cite{reinert2019approximating} measured
the closeness of the ERGM \cref{ERGM} in the subcritical region and the
Erd\H{o}s--R\'enyi random graph $G(n,p)$ in Wasserstein distance with respect to the Hamming metric. To state their result, we identify a simple graph $G$ on $n$ labeled vertices ${1,\dots, n}$ with an element $x=(x_{ij})_{1\leq i<j\leq n}\in \{0,1\}^{\mathcal{I}}$, where $\mathcal{I}:=\{(i,j): 1\leq i<j\leq n\}$ and $x_{ij}=1$ if and only if there is an edge between vertices $i$ and $j$. The correspondence between $G$ and $x$ is one-to-one. 
Similarly, a random graph corresponds to a random element in $\{0,1\}^{\mcl{I}}$.
In this manner, the ERGM \eqref{ERGM} induces a random element $Y\in \{0,1\}^{\mathcal{I}}$. Similarly, $G(n,p)$ induces a random element $X$ in $\{0,1\}^{\mathcal{I}}$. Let $h: \{0,1\}^{\mathcal{I}}\to \mathbb{R}$ be a test function. For $x\in \{0,1\}^{\mathcal{I}}$ and $s\in {\mathcal{I}}$, define $x^{(s,1)}$ to have 1 in the $s$th coordinate and otherwise the same as $x$, and define $x^{(s,0)}$
similarly except there is a 0 in the $s$th coordinate. Define
\begin{equation}\label{deltah}
\Delta_s h(x):=h(x^{(s,1)})-h(x^{(s,0)}), \quad \|\Delta h\|:=\sup_{x\in \{0,1\}^{\mcl{I}}, s\in \mcl{I}}|\Delta_s h(x)|.
\end{equation}
\cite[Theorems~1.13]{reinert2019approximating} proved that in the subcritical region of parameters introduced in the previous paragraph,
we have
\begin{equation}\label{rect5}
|\E h(Y)-\E h(X)|\leq C  \|\Delta h\| n^{3/2},
\end{equation}
where $C:=C(\beta, H)$ is a positive constant depending only on the parameters $\beta_1,\dots, \beta_k$ and the subgraphs $H_1,\dots, H_k$ in the definition of the ERGM.
Choosing $h(x)=\sum_{1\leq i<j\leq n}x_{ij}$ (so that $\|\Delta h\|=1$), we obtain
\begin{equation}\label{eq:LLN}
\left|\frac{\E\sum_{1\leq i<j\leq n} Y_{ij}}{{n\choose 2}}-p\right| \leq \frac{C}{\sqrt{n}}.
\end{equation}

Because the total number of edges in $G(n,p)$ satisfies a CLT (normal approximation of binomial distributions), it is natural to conjecture that the same holds for ERGM in the subcritical region.
However, 
\cref{rect5} can not give a CLT for the total number of edges in ERGM, because the error rate in \cref{rect5} is not small enough and we can not standardized $\sum_{1\leq i<j\leq n} Y_{ij}$ and $\sum_{1\leq i<j\leq n} X_{ij}$ in the same way (c.f. \cref{eq:edgecountstan}). On the other hand, the CLT
 was proven in the special case of two-star ERGMs in the subcritical region (where $k=2$, $\beta_2>0$ in \eqref{ERGM} and $H_2$ is a two-star, i.e., a graph with three vertices and two edges connecting them) by \cite{mukherjee2023statistics}. 
They used an explicit relation between the number of two-stars and the degrees of vertices to prove the CLT (see also \cite{park2004statistical}).
\cite{bianchi2024limit} proved the CLT for the edge-triangle model, although we do not understand how they justified the interchange of limit and differentiation for the free energy, which is crucial for their proof to work.
Partial attempts to analyze general ERGMs in the subcritical region were made by
\cite[Theorem~2]{ganguly2019sub}, who proved a CLT for the number of edges in $o(n^2)$ disconnected locations in the graph using the method of moments. 
\cite{sambale2020logarithmic} showed that if a CLT for the number of edges can be proved, then a CLT for general subgraph counts follows as a consequence for the ERGM in Dobrushin's uniqueness region, that is, when (recall the definition of $\Phi_\beta$ in \cref{varph1})
\ben{\label{eq:Dobrushin}
\Phi_{\beta}'(1)<2.
}
However, whether CLT holds for the number of edges for general ERGMs is an open problem. 

In this paper, we employ Stein's method (\cite{stein1972bound}) to address this problem and prove the CLT with a non-asymptotic error bound 
(cf. \cref{thm:ERGM}). As a byproduct, we also derive an explicit formula for the asymptotic variance (cf. \cref{eq:sigmasq}) and the CLT for general subgraph counts (cf. \cref{cor:general}). 

To prove CLT for ERGM, we consider a more general statistical physics model. Assume the joint density of $N$ particles, $\{Y_1,\dots, Y_N\}$, is given in the form of 
\ben{\label{eq:nonlinear}
\exp(g(y_1,\dots, y_N))\prod_{i=1}^N f_i(y_i),
}
where $f_i$ is the marginal distribution of the particle $Y_i$ at infinite temperature (corresponding to the exponential factor equal to one), and $g$ represents the interactions among these particles. ERGM \cref{ERGM} is a special case of \cref{eq:nonlinear}. See \cref{ex:CW} for another example.

In \cref{sec:cltexp}, we develop a CLT for general functionals of nonlinear exponential families as given by \cref{eq:nonlinear} by generalizing the approach of \cite{chatterjee2008new}. 
In \cref{sec:ERGM}, we apply the general CLT to the ERGM.
Our general CLT may be useful for other models as well, see, for example, \cref{ex:CW}.

\section{CLT for nonlinear exponential families}\label{sec:cltexp}
In this section, we develop a CLT for functionals of nonlinear exponential families \cref{eq:nonlinear}.
Let $N\geq 1$ be an integer and $X=(X_{1}, \ldots ,X_{N})$ be independent random variables (each $X_i$ has the baseline distribution $f_i$ in \cref{eq:nonlinear}). Let $Y=(Y_{1}, \ldots ,Y_{N})$ be a random vector following the distribution
\begin{equation}
   \label{eq:1}
  \begin{aligned}
     \P(Y=dy) = \frac{h(y)}{\E h(X)} \P(X=dy),\ y\in \IR^N,
  \end{aligned}
\end{equation}
where $h(x)=\exp\left\{ g(x) \right\}>0$ for a measurable function $g: \mathbb{R}^{N}\rightarrow \mathbb{R}$ such that $\E h(X)<\infty$. Let $f: \mathbb{R}^{N}\rightarrow \mathbb{R}$ be a function satisfying 
\begin{equation}
   \label{eq:2}
  \begin{aligned}
  \E f(Y)=0.
  \end{aligned}
\end{equation}
Let $W:=f(Y)$. Assume without loss of generality that $W$ is appropriately normalized to have a variance close to 1. 
To measure the distributional distance between $\mcl{L}(W)$ and the standard normal distribution $N(0,1)$, we consider 
the  Kolmogorov distance
\begin{equation}
   \label{eq:dk}
  \begin{aligned}
  d_{\text{Kol}}(W, Z):=\sup_{x\in \mathbb{R}} |\P(W\leq x)-\P(Z\leq x)|
  \end{aligned}
\end{equation}
and
the Wasserstein distance
\ben{\label{eq:d2dist}
d_{\text{Wass}}(W, Z):=\sup_{\psi: \|\psi'\|_\infty\leq 1} |\E \psi(W)-\E \psi(Z)|,
}
where $Z\sim N(0,1)$, the sup in \cref{eq:d2dist} is taken over all absolutely continuous functions $\psi:\IR\to \IR$ and $\|\cdot\|_\infty$ denotes the $L^\infty$ norm. 

As in typically applications of Stein's method, to exploit the dependency structure of $W$, we construct a small perturbation of $W$. Our construction is motivated by the work of \cite{chatterjee2008new}, Construction 4A of \cite{chen2010stein} and \cite{shao2024berry}, although they only considered functionals of independent random variables (see \cref{rem:theorem1(2)} for some new features of our more general setting). 
Let $X'=(X_{1}', \ldots ,X_{N}')$ be an independent copy of $X$.
Let 
\ben{\label{eq:X[i]}
X^{[i]}=(X_{1}, \ldots ,X_{i}, X_{i+1}', \ldots, X_{N}')
}
and 
\ben{\label{eq:X(i)}
X^{(i)}=(X_{1}, \ldots ,X_{i-1}, X_{i}', X_{i+1}, \ldots, X_{N}).
}
We first introduce the following quantities. These quantities involve moments and conditional expectations of $f$ and $g$ under the perturbation from $X$ to $X'$. 

Define
\begin{equation}
    \label{eq:t4}
   \begin{aligned}
   \Delta_{1,i}(X) =   \frac{1}{2}\E \left[ \left(f(X) -f(X^{(i)})\right)
    \left(f(X^{[i]})-f(X^{[i-1]})\right)\bigg| X\right],
   \end{aligned}
 \end{equation}
 and
 \begin{equation}
    \label{eq:t5}
   \begin{aligned}
  \Delta_{2,i}(X)= \frac{1}{2}\E \left[\left(g(X)- g(X^{(i)})\right)
    \left(f(X^{[i]})-f(X^{[i-1]})\right)\bigg|X\right]. 
   \end{aligned}
 \end{equation}
Let $D^*_{i}(X,X')$ be any symmetric function
of $X$ and $X'$ such that 
\ben{\label{eq:Dstar}
D^*_{i}(X,X')=D^*_{i}(X',X)\geq \lvert f(X^{[i]})-f(X^{[i-1]}) \rvert.
}
Let 
\ben{\label{eq:b}
 a:=\E h(X),\quad b=\E \left[\sum_{i=1}^N \Delta_{1,i}(Y)\right], \quad \text{(we assume $b\ne 0$)}.
 }
As we will see in applications below, $b$ is typically of a constant order (cf. \cref{eq:2007,eq:basym}), although not equal to $\Var(W)$ in general. 
Finally, let
 \begin{equation}
    \label{eq:t1}
   \begin{aligned}
     \delta_{1}=& \frac{1}{a}\sum_{ i=1 }^{ N } \E \Big\{ h(X) (f(X)- f(X^{(i)}))^2 |f(X^{[i]})-f(X^{[i-1]}| \Big\}\\
     &+\frac{1}{a}\sum_{ i=1 }^{ N } \E \Big\{ h(X)  \exp\left[ |g(X)- g(X^{(i)})| \right] (g(X)-g(X^{(i)}))^2\\&\qquad\qquad \times \left(\lvert g(X)- g(X^{(i)})\rvert+ \lvert f(X)- f(X^{(i)})\rvert\right) \left\lvert f(X^{[i]})-f(X^{[i-1]})\right\rvert \Big\}, 
   \end{aligned}
 \end{equation}
\begin{equation}
    \label{eq:t1'}
   \begin{aligned}
     \delta_{1}'=& \frac{1}{a}\sum_{ i=1 }^{ N } \E\left\{ h(X) \exp\left[ |g(X)- g(X^{(i)})| \right] D^*_{i}(X,X') \left\lvert f(X)- f(X^{(i)})\right\rvert |g(X)- g(X^{(i)})|\right\}
                 \\&+ \frac{1}{a}\E\left\{ h(X) \Big\lvert\sum_{ i=1 }^{ N } \E \big[ D^*_{i}(X,X')(f(X)- f(X^{(i)}))|X\big]\Big\rvert\right\}\\
     &+\frac{1}{a}\sum_{ i=1 }^{ N } \E \Big\{ h(X)  \exp\left[ |g(X)- g(X^{(i)})| \right] (g(X)-g(X^{(i)}))^2\\&\qquad\qquad \times \left(\lvert g(X)- g(X^{(i)})\rvert+ \lvert f(X)- f(X^{(i)})\rvert\right) \left\lvert f(X^{[i]})-f(X^{[i-1]})\right\rvert \Big\},
   \end{aligned}
 \end{equation}
 \begin{equation}
    \label{eq:t2}
   \begin{aligned}
 \delta_{2}  &= \sqrt{\Var(\sum_{ i=1 }^{ N } \Delta_{1,i}(Y) )} ,
   \end{aligned}
 \end{equation}
\begin{equation}
   \label{eq:t3}
  \begin{aligned}
 \delta_{3}=   \sqrt{\Var \left\{ \sum_{ i=1 }^{ N } \Delta_{2,i}(Y) -(1-b) f(Y)\right\}}.
  \end{aligned}
\end{equation}
The following is our general CLT, which is a generalization of the result for functionals of independent random variables by \cite{chatterjee2008new}. 
We defer the proof to \cref{sec:proofCLT}.

\begin{theorem}\label{theorem:1}
For $W=f(Y)$ defined above and for the distributional distances $d_{\text{Wass}}$ and $d_{\text{Kol}}$ defined in \cref{eq:d2dist} and \cref{eq:dk} respectively, we have
 \begin{equation}
    \label{eq:3}
   \begin{aligned}
d_{\text{Wass}}(W, Z)\leq \frac{C}{|b|}( \delta_{1}+\delta_{2}+
\delta_{3}),
   \end{aligned}
 \end{equation}
 \begin{equation}
    \label{eq:ergmdk}
   \begin{aligned}
   d_{\text{Kol}}(W, Z)\leq \frac{C}{|b|}( \delta_{1}'+\delta_{2}+
\delta_{3}),
   \end{aligned}
 \end{equation}
 where $C$ is an absolute constant. 
 \end{theorem}

\begin{remark}\label{rem:theorem1}
If $g(x)=0$, then the problem reduces to the normal approximation of functionals of independent random variables studied by \cite{chatterjee2008new}. In this case, $h(X)=\E h(X)=a=1$, $g(X)=g(X^{(i)})$, $b=1$ (if $\Var(W)=1$), $\delta_3=0$ and the second term in $\delta_1$ equals 0. Then, our Wasserstein bound simplifies to a similar bound to \cite[Theorem~2.2]{chatterjee2008new}, except that we use the fixed order interpolation instead of a random order interpolation (cf. Constructions 4A and 4B in \cite{chen2010stein} and \cite{shao2024berry}).

As explained in \cite{chatterjee2008new}, for the classical case of standardized sum of i.i.d.\ random variables (where $f$ is a summation and $f(X)-f(X^{(i)})=f(X^{[i]})-f(X^{[i-1]})=(X_i-X_i')/\sigma$ for a constant $\sigma\asymp \sqrt{N}$), the bound is of the optimal order $O(1/\sqrt{N})$, provided that each $X_i$ has a finite fourth moment.

Our Kolmogorov bound in this special case of $g(x)=0$, obtained by adapting the method of \cite{shaozhang2019}, is
also of the optimal order $O(1/\sqrt{N})$ for sums of i.i.d.\ random variables under the finite fourth moment assumption. 
\end{remark}

\begin{remark}\label{rem:theorem1(2)}

If $g(x)\neq 0$, our result and its proof reveal some interesting new features of the problem. First, due to the deviation of the distribution of $Y$ from $X$, $b$ in general no longer equals the variance of $W$. We need to develop a new recursive argument in Stein's method (cf. \cref{eq:outline}) to address such a discrepancy. Second, additional terms appear in $\delta_1$, $\delta_1'$, and $\delta_3$ involving the difference $g(X)-g(X^{(i)})$ (the influence of each $X_i$ on the function $g$ governing the exponential change of measure). Intuitively, we need to control such differences to prevent wild behavior of the exponential change of measure. Third, a symmetry argument is employed to gain one additional factor of $|g(X)-g(X^{(i)})|$ in the last terms of $\delta_1$ and $\delta_1'$ (cf. \cref{eq:23}). This additional factor is crucial for obtaining a vanishing error bound for the application to the ERGM.
\end{remark}

\begin{example}[Curie--Weiss model]
  \label{ex:CW}
To give a simple example to illustrate the application of \cref{theorem:1}, we consider the Curie--Weiss model without external field in the subcritical region. This model is well-studied in the literature (cf. \cite{ellis1978limit,ellis1978statistics}) and the exchangeable pairs approach in Stein's method also works for this model (cf. \cite{chatterjee2011nonnormal,chen2013stein}).


Let $X_1,\dots, X_N$ be i.i.d.\ with distribution
$P\left(X_i= \pm 1\right)=1 / 2$. Let $0<\beta<1$ be a fixed parameter (inverse temperature). 
Let
$$
\begin{gathered}
h(x)=\exp \left(\frac{\beta}{2 N} s^2(x)\right),\quad  g(x)=\frac{\beta}{2 N} s^2(x),\\
s(x)=x_1+\cdots+x_N, \quad f(x)=\frac{s(x)}{\sigma_{N}}, \quad \sigma_{N}^2=\frac{N}{1-\beta} .
\end{gathered}
$$
Let $Y=(Y_{1},Y_{2}, \ldots ,Y_{N})$ have the following distribution
\begin{equation*}
   \label{eq:2011}
  \begin{aligned}
     \P(Y=dy) = \frac{h(y)}{\E h(X)} \P(X=dy),\ y\in \IR^N,
  \end{aligned}
\end{equation*}
and let
\be{
W=\frac{\sum_{i=1}^N Y_i}{\sigma_N}.
}
where $\sigma_N^2$ is the asymptotic variance of $\sum_{i=1}^N Y_i$. It is known that $W$ converges in distribution to $N(0,1)$ as $N\to \infty$.
Using \cref{theorem:1}, we can prove this CLT together with the optimal rate of convergence $O(1/\sqrt{N})$ in both the Wasserstein and the Kolmogorov distances. We defer the details to \cref{sec:CW}.


\end{example}

\section{Normal approximation for ERGM}\label{sec:ERGM}

To prepare for the statement of our main result for the ERGM, 
let $Y=\{Y_{ij}: 1\leq i<j\leq n\}$ be edge indicators of the ERGM \cref{ERGM}, i.e., $Y_{ij}=1$ if there is an edge connecting the vertices $i$ and $j$ and $Y_{ij}=0$ otherwise. In other words, the joint probability mass function of $Y$ is
\be{
p_\beta(y)=\frac{1}{Z(\beta)} \exp\left\{\sum_{j=1}^k \frac{\beta_j}{n^{|\mathcal{V}(H_j)|-2}}  |\text{Hom}(H_j, G_y)|\right\},\quad y\in \{0,1\}^\mcl{I},
}
where $\mcl{I}=\{(i,j):1\leq i<j\leq n\}$ and $G_y$ is the graph with edge indicators $y$. Label the vertices of $H_{j}$ by $l_1,l_2,\ldots l_{v_j}$ and let $\mathcal{E}(H_j)$ be the set of its edges, then 
(
let $y_{ji}=y_{ij}$ for $i<j$), 
\begin{equation*}
  \begin{aligned}
    |\text{Hom}(H_{j},G_y)|= \sum_{\begin{subarray}{c} k_{1},k_{2}, \ldots ,k_{v_{j}}\in \{1,\dots, n\} \\ k_{1},k_{2}, \ldots ,k_{v_{j}} \text{ are distinct}\end{subarray}} \prod_{\{ l_p,l_q \}\in \mathcal{E}(H_{j}) } y_{k_{p},k_{q}}.
  \end{aligned}
\end{equation*}
Let
\ben{\label{eq:edgecountstan}
W=W_n=\frac{\sum_{1\leq i<j\leq n}Y_{ij}-\mu_n}{\sigma_n},\quad \mu_n=\E\sum_{1\leq i<j\leq n} Y_{ij}
} 
and \ben{\label{eq:sigmasq}
\sigma_n^2=\sigma_n^2(\beta)=\frac{Np(1-p)}{1-\sum_{j=2}^k \beta_j e_j(e_j-1) 2p^{e_j-1}(1-p)},\quad N={n\choose 2}.
}
Note that for the denominator (recall $e_j\geq 2$ for $j=2,\dots, k$ and \cref{varph1}), we have
\begin{equation}
    \label{3001}
    \sum_{j=2}^k \beta_j e_j(e_j-1) 2p^{e_j-1}(1-p)= 2 p(1-p) \Phi_\beta'(p).
\end{equation}
On the other hand, by \cref{varph1,eq:p}
\begin{equation}\label{3002}
  \varphi_\beta'(p)= \frac{e^{2 \Phi_\beta(p)}}{(e^{2 \Phi_\beta(p)}+1)^2} 2 \Phi_\beta'(p)= 2 p(1-p) \Phi_\beta'(p),
\end{equation}
which is $<1$ in subcritical region \cref{subcritical}. Therefore, by \cref{3001,3002}, $\sigma_n^2$ is well defined and is of the order $\asymp n^2$.
For the case $k=2$ and $H_2$ is a two-star, the expression of variance \cref{eq:sigmasq} coincides with that in \cite[Theorem~1.4]{mukherjee2023statistics}. They derived it using explicit computations for the two-star ERGM, while our general expression arises naturally in the application of Stein's method (cf. \cref{eq:1006}).

The following is our main result for the ERGM.
We defer its proof to \cref{sec:proofERGM}.

\begin{theorem}\label{thm:ERGM} 
For the ERGM \cref{ERGM} in the subcritical region \cref{subcritical}, the normalized number of edges \cref{eq:edgecountstan} is asymptotically normal as $n\to \infty$ with error bounds
\ben{\label{eq:thmsub}
  d_{\text{Wass}}(W, Z)\leq \frac{C}{n^{1/4}} \text{ and } d_{\text{Kol}}(W, Z)\leq \frac{C}{n^{1/4}},
}
where $Z\sim N(0,1)$ and
$C:=C(\beta, H)$ is a positive constant depending only on the parameters $\beta_1,\dots, \beta_k$ and the subgraphs $H_1,\dots, H_k$ in the definition of the ERGM.
Moreover, in Dobrushin's uniqueness region \cref{eq:Dobrushin}, we have
\ben{\label{eq:thmdob}
  d_{\text{Wass}}(W, Z)\leq \frac{C}{\sqrt{n}} \text{ and } d_{\text{Kol}}(W, Z)\leq \frac{C}{\sqrt{n}}.
}
\end{theorem}

\begin{remark}\label{rem:1}
\cref{thm:ERGM} leaves the following two problems open. 

1. From \cite[Theorem 1]{ganguly2019sub}, $W_n$ satisfies a Gaussian concentration inequality and the third absolute moments of $\{W_n\}_{n=1}^\infty$ are uniformly bounded in $n$. 
This implies the sequence $W_n^2$ is uniformly integrable. Together with the CLT for $W_n$, we have $\E W_n^2\to 1$ and
$\sigma_n^2$ must be the asymptotic variance of $\sum_{1\leq i<j\leq n}Y_{ij}$. 
However, 
whether $\mu_n/{n\choose 2}-p=O(1/n)$ remains open (this is a faster rate than that of \cref{eq:LLN}). This fact was proved for the two-star ERGM by \cite{mukherjee2023statistics} using an explicit computation.


2. The subcritical region for possibly negative values of $\beta_2,\dots, \beta_k$ is less well understood. See \cite[Sections~6 and 7]{chatterjee2013estimating} for some partial results. Dobrushin's uniqueness region \cref{eq:Dobrushin} can be similarly defined for this case, simply changing $\beta$ to their absolute values (cf. \cite{sambale2020logarithmic}). However, we need to restrict to positive $\beta_2,\dots, \beta_k$ because we use the results of \cite{ganguly2019sub} in \cref{eq:GN19}, which crucially rely on $\beta_2,\dots, \beta_k$ being positive. If \cref{eq:GN19}, even with a slower rate, can be proved for possibly negative values of $\beta_2,\dots, \beta_k$, then our result can be extended to that case.
\end{remark}

\begin{corollary}[CLT for general subgraph counts]\label{cor:general}
Under the same setting as \cref{thm:ERGM},
     let $H$ be a subgraph with $v$ vertices and $e$ edges and let $$W_H=\frac{|\text{Hom}(H,G_Y)|-\E |\text{Hom}(H,G_Y)|}{2 n^{v-2} e p^{e-1}\sigma_n}. $$ Then, in the subcritical region,
    $$ d_{\text{Wass}}(W_H,Z)\leq \frac{C}{n^{1/4}}.$$
\end{corollary}
\begin{proof}
Recall $W$ from \cref{eq:edgecountstan}.
    By the definition of the Wasserstein distance, we have
    \begin{equation}
        d^2_{\text{Wass}}(W_H,W)\leq \E (W-W_H)^2 \leq \frac{C n^{2v-\frac{5}{2}}}{n^{2v-4} \sigma_n^2} \leq \frac{C}{\sqrt{n}},
    \end{equation}
    where the second inequality follows from \cref{lemma:1002} in \cref{sec:proofERGM} and the last inequality follows from \cref{eq:sigmasq}. Thus, by \cref{thm:ERGM} and the triangle inequality for the Wasserstein distance,
    \begin{equation}
        d_{\text{Wass}}(W_H,Z)\leq d_{\text{Wass}}(W_H,W)+d_{\text{Wass}}(Z,W)\leq \frac{C}{n^{1/4}}.
    \end{equation}
\end{proof}


\section{Proof of Theorem \ref{theorem:1}}\label{sec:proofCLT}
\begin{proof}[Proof of \cref{theorem:1}]
We first consider the Wasserstein distance. Let $F$ be the bounded solution to the Stein equation
\ben{\label{eq:steineq}
w F(w)-F'(w)=\psi(w)-\E \psi(Z),\quad \forall\ w\in \IR.
}
If $\|\psi'\|_\infty\leq 1$, it is known that $F$ satisfies
\ben{\label{eq:solproperty}
\|F\|_\infty, \|F'\|_\infty, \|F''\|_\infty\leq 2.
}
See, for example, \cite[Lemma~2.4]{chen2010normal}.
From \cref{eq:steineq} and the definition of $d_{\text{Wass}}$ in \cref{eq:d2dist}, we have
\ben{\label{eq:bound1}
d_{\text{Wass}}(W, Z)\leq \sup_{F} |\E [W F(W)]-\E F'(W)|,
}
where the sup is taken over all functions $F$ satisfying \cref{eq:solproperty}.
A typical application of Stein's method proceeds to bound the right-hand side of \cref{eq:bound1} using the structure of $W$ and Taylor's expansion.

An important fact relating the expectation with respect to $Y$ and that to $X$ we use frequently below is that from \cref{eq:1}, for functions $G: \mathbb{R}^N\to \mathbb{R}$ such that $\mathbb{E} |h(X)G(X)|<\infty$, we have
\ben{\label{eq:transfer}
\E G(Y)=\frac{1}{a}\E [h(X)G(X)],\quad \text{where}\ a:=\E h(X).
}
Fix a function $F:\IR\to \IR$ satisfying \cref{eq:solproperty}.
From \cref{eq:transfer}, we have 
\begin{equation}
   \label{eq:4}
  \begin{aligned}
    \E [f(Y)F(f(Y))]=\frac{\E [h(X) f(X)F(f(X))]}{\E h(X)}= I+II,
  \end{aligned}
\end{equation}
where  
\begin{equation}
   \label{eq:5}
  \begin{aligned}
  I:=\frac{\E \big[(f(X)-\E f(X)) h(X)F(f(X))\big]}{\E h(X)}
  \end{aligned}
\end{equation}
and 
\begin{equation}
   \label{eq:6}
  \begin{aligned}
     II: = \frac{\E  [h(X)F(f(X))]}{\E h(X)} \E f(X)=\E F (f(Y)) \E f(X).
  \end{aligned}
\end{equation}

{\bf Outline of the proof.} Unlike a typical application of Stein's method, because of the deviation of the distribution of $Y$ from $X$, our proof structure is more involved. We first give some lemmas which make the proof more clear.
\begin{lemma}\label{lemma:I}
For  $I$ defined in \cref{eq:5}, we have 
\begin{equation}
   \label{eq:I1andI2}
  \begin{aligned}
  I=I_{1}+I_{2}+ O(1)\delta_{1},
  \end{aligned}
\end{equation}
where $\delta_{1}$ is defined in \cref{eq:t1},
\be{
I_1=\frac{1}{2a} \sum_{ i=1 }^{ N } \E \Big[ h(X)F'(f(X))\big(f(X)-f(X^{(i)})\big)
    \big(f(X^{[i]})-f(X^{[i-1]})\big)\Big],
}
\be{
I_2=\frac{1}{2a} \sum_{ i=1 }^{ N } \E \Big[ h(X) F(f(X)) \big(g(X)- g(X^{(i)})\big)
    \big(f(X^{[i]})-f(X^{[i-1]})\big)\Big].
}
\end{lemma}
 
\begin{lemma}\label{lemma:I1}
For $I_{1}$ in \cref{eq:I1andI2}, 
\begin{equation}
   \label{eq:I1}
  \begin{aligned}
  I_{1}= b \E F'(f(Y)) + O(1)\delta_{2},
  \end{aligned}
\end{equation}
where $\delta_{2}$ is defined in \cref{eq:t2}.
\end{lemma}
\begin{lemma}\label{lemma:I2+II}
For $II$ in \cref{eq:6} and $I_{2}$ in \cref{eq:I1andI2}, 
\begin{equation}
   \label{eq:I2+II}
  \begin{aligned}
  I_{2}+II = O(1)(\delta_{1}+\delta_{3}) + (1-b) \E f(Y) F(f(Y)),
  \end{aligned}
\end{equation}
where $b$, $\delta_{1}$ and $\delta_{3}$ are defined in \cref{eq:b,eq:t1,eq:t3} respectively.
\end{lemma}

Now, by \cref{lemma:I,lemma:I1,lemma:I2+II}, we obtain
\besn{\label{eq:outline}
&\E [W F(W)]-\E F'(W)=\E [f(Y) F(f(Y))]-\E F'(f(Y))\\
=&I+II-\E F'(f(Y))\\
= & I_1+I_2+II-\E F'(f(Y))+O(1)\delta_1\\
= & (1-b) \{\E[W F(W)]-\E F'(W)\}+O(1)(\delta_1+\delta_2+\delta_3).
}
Solving the recursive equation for $\E[W F(W)]-\E F'(W)$ leads to \cref{eq:3}.
It now remains to prove \cref{lemma:I,lemma:I1,lemma:I2+II}.

\bigskip

We proceed by adapting the approach of \cite{chatterjee2008new} in Stein's method (see \cite{chen2010stein} for the variation of the approach that we use), who only considered functionals of independent random variables. 
\begin{proof}[Proof of \cref{lemma:I}]

Recall that $X'=(X_1',\dots, X_n')$ is an independent copy of $X=(X_1,\dots, X_n)$. 
We will use the fact that if $G: \IR^2\to \IR$ satisfies $G(x, x')=G(x', x), \forall\ x, x'\in \IR$ and $\E |G(X_i, X_i')|<\infty$, then
\ben{\label{eq:symmetry}
\E G(X_i, X_i')=\E G(X_i', X_i)=\frac{1}{2} \E [G(X_i, X_i')+G(X_i',X_i)].
}
Recall the definitions of $X^{[i]}$ and $X^{(i)}$ in \cref{eq:X[i]} and \cref{eq:X(i)}, respectively, and $a=\mathbb{E} h(X)$. 

For $I$, using the telescoping sum in the first equation and \cref{eq:symmetry} in the second and fourth equations, we obtain
  \begin{align}
   \label{eq:7}
 I=&\frac{1}{a} \E \Big[ h(X) F (f(X)) \sum_{ i=1 }^{ N } (f(X^{[i]})-f(X^{[i-1]}))\Big]\nonumber \\
    =&\frac{1}{2a} \sum_{ i=1 }^{ N }\E \Big[ \{h(X) F (f(X))-h(X^{(i)}) F (f(X^{(i)}))\}( f(X^{[i]})-f(X^{[i-1]})) \Big]\nonumber \\
    =& \frac{1}{2a} \sum_{ i=1 }^{ N }\E  \Big[ h(X^{(i)})  \{F (f(X))-F(f(X^{(i)}))\}( f(X^{[i]})-f(X^{[i-1]})) \Big]\nonumber \\
    &+ \frac{1}{2a} \sum_{ i=1 }^{ N }\E \Big[ \{h(X) -h(X^{(i)})\} F (f(X))( f(X^{[i]})-f(X^{[i-1]}))\Big] \nonumber \\
    =& \frac{1}{2a} \sum_{ i=1 }^{ N }\E  \Big[h(X)  \big\{ F (f(X^{(i)}))-F(f(X))\big\}( f(X^{[i-1]})-f(X^{[i]}))\Big]\nonumber  \\
    &+ \frac{1}{2a} \sum_{ i=1 }^{ N }\E  \Big[\{h(X) -h(X^{(i)})\} F (f(X))( f(X^{[i]})-f(X^{[i-1]}))\Big] \nonumber \\
=:&I_{1}+I_{2}+R_{1,1}+R_{1,2},  
\end{align}
where, using Taylor's expansion (recall $h(x)=\exp\{g(x)\}$),
\begin{equation}
   \label{eq:20}
  \begin{aligned}
    R_{1,1}&=- \frac{1}{2a} \sum_{ i=1 }^{ N }  \E \Big[ h(X)F''\left(f(X)+U(f(X^{(i)})-f(X))\right)\{f(X)-f(X^{(i)})\}^{2}
    \\&\qquad\qquad \qquad \qquad\times\{f(X^{[i]})-f(X^{[i-1]})\}(1-U)\Big],
  \end{aligned}
\end{equation}
$U$ is a uniform random variable in $[0,1]$ independent of any other random variables, and 
\begin{equation}
   \label{eq:21}
  \begin{aligned}
    R_{1,2}&=-\frac{1}{2a} \sum_{ i=1 }^{ N } \E \Big[ \exp\big\{  g(X)+U(g(X^{(i)})-g(X))\big\}F\left(f(X)\right)\{g(X)-g(X^{(i)})\}^{2}
    \\&\qquad\qquad \qquad \qquad \times\{f(X^{[i]})-f(X^{[i-1]})\}(1-U) \Big].
  \end{aligned}
\end{equation}
For $R_{1,1}$ in \cref{eq:20}, 
we have
\ben{\label{eq:22}
R_{1,1}=O(1)\|F''\|_\infty \frac{1}{a} \sum_{i=1}^N \E \Big\{h(X)(f(X)-f(X^{(i)}))^2 |f(X^{[i]})-f(X^{[i-1]})|   \Big\},
}
where $O(1)$ denotes a constant such that $|O(1)|\leq C$.

For $R_{1,2}$ in \cref{eq:21}, by \cref{eq:symmetry} and Taylor's expansion, 
\begin{equation}
   \label{eq:23}
  \begin{aligned}
  R_{1,2}= & -\frac{1}{4a}\sum_{i=1}^N \E\Big\{\Big[ \exp\big(g(X)+U(g(X^{(i)})-g(X)) \big) F(f(X))\\
  &\qquad\qquad \qquad \qquad-\exp\big( g(X^{(i)})+U(g(X)-g(X^{(i)}))\big)F(f(X^{(i)}))\Big]\\
  &\qquad\qquad  \qquad\times(g(X)-g(X^{(i)}))^2 (f(X^{[i]})-f(X^{[i-1]}))(1-U) \Big\}\\
  =&O(1) \| F \|_{\infty} \frac{1}{a}\sum_{i=1}^N\E \Big\{  h(X)  \exp\big[ \xi(g(X)- g(X^{(i)})) \big]\big\lvert g(X)- g(X^{(i)})\big\rvert^{3} 
       \big\lvert f(X^{[i]})-f(X^{[i-1]})\big\rvert\Big\}\\
       &\quad +O(1) \| F' \|_{\infty} \frac{1}{a}\sum_{i=1}^N\E\Big\{   h(X)  \exp\big[ \xi'(g(X)- g(X^{(i)})) \big] \big\lvert g(X)- g(X^{(i)})\big\rvert^{2}\\&\qquad \quad \qquad \qquad \qquad \times\big\lvert f(X)-f(X^{(i)})\big\rvert
       \big\lvert f(X^{[i]})-f(X^{[i-1]})\big\rvert \Big\},
  \end{aligned}
\end{equation}
where $\xi$ and $\xi'$ are random variables supported on $[-1,0]$. We remark that the above symmetry argument is crucial to have a vanishing error for the application to ERGM. See \cite{fang2021high,fang2022new} for applications of this symmetry argument in multivariate normal approximations.

From \cref{eq:22,eq:23}, we have (recall the property of $F$ in \cref{eq:solproperty}) 
\begin{equation}
   \label{eq:24}
  \begin{aligned}
  |R_{1,1}|+|R_{1,2}|\leq C\delta_1 \ \text{(recall \cref{eq:t1})}.
  \end{aligned}
\end{equation}
\end{proof}

\begin{proof}[Proof of \cref{lemma:I1}]
For $I_{1}$, taking conditional expectation given $X$, using the definition of $\Delta_{1,i}(\cdot)$ in \cref{eq:t4}, and using an appropriate centering (recall \cref{eq:transfer}),
\begin{equation}
   \label{eq:10}
  \begin{aligned}
    I_{1}=&\frac{1}{a} \E\left[  h(X)F'(f(X)) \left\{\sum_{ i=1 }^{ N } \Delta_{1,i}(Y)-b\right\}\right] + b\E F'(f(Y))\\
          =:& R_{2}+  b\E F'(f(Y))
    \end{aligned}
\end{equation}
where (recall \cref{eq:b}, \cref{eq:t4} and \cref{eq:transfer})
\be{
b=\E\left[ \sum_{i=1}^N \Delta_{1,i}(Y)\right] \ne 0,
}
\begin{equation}
   \label{eq:11}
  \begin{aligned}
    R_{2} &= O(1) \| F' \|_{\infty}\frac{1}{a}  \E\left[  h(X) \left\lvert\frac{1}{2}\sum_{ i=1 }^{ N } \E \left[ \left(f(X) -f(X^{(i)})\right)
    \left(f(X^{[i]})-f(X^{[i-1]})\right)\bigg| X\right]-b\right\rvert\right]\\
    &= O(1) \| F' \|_{\infty} \frac{1}{a}  \E \left[ h(X) \left\lvert \sum_{ i=1 }^{ N } \Delta_{1,i}(X)-b\right\rvert\right]\\
    &= O(1) \| F' \|_{\infty}   \E  \left\lvert \sum_{ i=1 }^{ N } \Delta_{1,i}(Y)-b\right\rvert=O(1) \| F' \|_{\infty}\sqrt{\Var(\sum_{ i=1 }^{ N } \Delta_{1,i}(Y) )}.
  \end{aligned}
\end{equation}
Combining \cref{eq:solproperty,eq:11}, we obtain $R_{2}= O(1) \delta_{2}$.
\end{proof}

\begin{proof}[Proof of \cref{lemma:I2+II}]
For $I_{2}$, taking conditional expectation with respect to $X$ and recalling \cref{eq:t5}, we obtain
\begin{equation}
   \label{eq:12}
  \begin{aligned}
    I_{2}=&\frac{1}{2a} \sum_{ i=1 }^{ N } \E\left[  h(X) F(f(X)) \left(g(X)- g(X^{(i)})\right)
    \left(f(X^{[i]})-f(X^{[i-1]})\right)        \right]
       \\=& \frac{1}{a} \sum_{ i=1 }^{ N } \E \left[ h(X) F(f(X)) \Delta_{2,i}(X)\right].
  \end{aligned}
\end{equation}

Next, to deal with $II=[\E F(f(Y))][\E f(X)]$ in \cref{eq:6}, we consider $\E f(X)$. Recalling $\E f(Y)=0$ and using \cref{eq:symmetry} in the fourth equation below, we have
\ban{
    \E f(X)=&\E f(X') =-\frac{1}{a} \E\left[ h(X) (f(X)-\E f(X'))\right]\nonumber \\
    =&-\frac{1}{a} \E \left[h(X) \sum_{ i=1 }^{ N } \left(f(X^{[i]}) - f(X^{[i-1]})\right)\right]\nonumber \\
    =&-\frac{1}{2a}   \sum_{ i=1 }^{ N } \E \left[\left(h(X)-h(X^{(i)})\right) \left(f(X^{[i]}) - f(X^{[i-1]})\right)\right]\nonumber \\
    =&-\frac{1}{2a}   \sum_{ i=1 }^{ N } \E\left[ h(X) \left(g(X)-g(X^{(i)})\right) \left(f(X^{[i]}) - f(X^{[i-1]})\right)\right]\nonumber \\
     &+\frac{1}{2a}   \sum_{ i=1 }^{ N } \E\left[ h(X) e^{\left\{ U( g(X^{(i)})-g(X) \right\}} \left(g(X)-g(X^{(i)})\right)^{2} \left(f(X^{[i]}) - f(X^{[i-1]})\right)(1-U)\right]\nonumber \\
     =&- \frac{1}{a}\sum_{ i=1 }^{ N }\E \left[h(X) \Delta_{2,i}(X)\right]+R_{4},\label{eq:13}
     }
where $U$ is a uniform random variable in $[0,1]$ independent of any other random variables and
\begin{equation}
   \label{eq:15}
  \begin{aligned}
  R_{4} = \frac{1}{2a}   \sum_{ i=1 }^{ N } \E \left[ h(X) e^{\left\{ U( g(X^{(i)})-g(X)) \right\}} \left(g(X)-g(X^{(i)})\right)^{2} \left(f(X^{[i]}) - f(X^{[i-1]})\right)(1-U)\right].
  \end{aligned}
\end{equation}

From \cref{eq:12,eq:13}, it follows that 
\begin{equation}
   \label{eq:14}
  \begin{aligned}
    I_{2}+II=& \frac{1}{a}  \E  \left[h(X) F(f(X)) \sum_{ i=1 }^{ N } \left(\Delta_{2,i}(X) - \frac{1}{a} \E h(X) \Delta_{2,i}(X)\right)\right]+ R_4\E F(f(Y)).  
    \end{aligned}
\end{equation}
Similar to bounding $R_{1,2}$, $R_4\E F(f(Y))$ can be bounded by $\delta_1$ defined in \cref{eq:t1}. 

Let
\begin{equation}
   \label{eq:16}
  \begin{aligned}
  R_{3}=   \E \left\{ F(f(Y)) \left[ \sum_{ i=1 }^{ N } \left(\Delta_{2,i}(Y) -  \E  \Delta_{2,i}(Y)\right)-(1-b) f(Y)\right]\right\}.
  \end{aligned}
\end{equation}
From the boundedness of $F$ from \cref{eq:solproperty} and recalling $
\E f(Y)=0$, we obtain $|R_3|\leq \delta_3$ (recall \cref{eq:t3}). 
\end{proof}


\bigskip

\noindent\textbf{Kolmogorov bound.} Next, we modify the above proof to obtain the bound \cref{eq:ergmdk} on the Kolmogorov distance. 
The main change is dealing with $I_1+R_{1,1}$.
We let $F_{x}$ be the bounded solution to the Stein equation
\begin{equation}
   \label{eq:10001}
  \begin{aligned}
  wF(w)-F'(w) = 1_{\{ w\leq x \}} -\P(Z\leq x),\quad \forall\ w\in \IR,
  \end{aligned}
\end{equation}
where $1_{\{\cdot\}}$ is the indicator function.
It is known that for any $x\in \mathbb{R}$, $F_{x}$ satisfies 
\begin{equation}
   \label{eq:10002}
  \begin{aligned}
  \| F_{x} \|_{\infty}, \|F'_{x}\|_{\infty}\leq 1.
  \end{aligned}
\end{equation}
See, for example, \cite[Lemma~2.3]{chen2010normal}. Compared with \cref{eq:solproperty}, the solution of \cref{eq:10001}, $F_{x}$, does not have a bounded second derivative. 
The second derivative was needed to control $R_{1,1}$ in \cref{eq:20}.
Therefore, in the proof of the Kolmogorov bound, we need to deal with $I_{1}+R_{1,1}$ in another way and the other terms in the above proof of the Wasserstein bound remain unchanged. We will adapt the method of \cite{shaozhang2019} to avoid $F_x''$ and prove that 
\begin{equation}
   \label{eq:10003}
  \begin{aligned}
  I_{1}+R_{1,1} = b \E F'_{x}(f(Y)) + O(1)\delta_{1}'.
  \end{aligned}
\end{equation}
The desired Kolmogorov bound then follows.

It remains to prove \cref{eq:10003}. In the remaining proof, we fixed $x\in \IR$ and write $F:=F_x$ for simplicity. Let $$D^{(i)}=D^{(i)}(X,X')= f(X)-f(X^{(i)}),$$ and 
\[
  D^{[i]}=D^{[i]}(X,X') = f(X^{[i-1]})-f(X^{[i]}).
\]
We now consider (recall \cref{eq:7}) 
\be{
I_{1}+R_{1,1}=\frac{1}{2a} \sum_{ i=1 }^{ N }\E  \Big[h(X)  \big\{ F (f(X^{(i)}))-F(f(X)\big\}( f(X^{[i-1]})-f(X^{[i]}))\Big].
}
By Taylor's expansion and recalling the definition of $b$ in \cref{eq:b}, we have
\ban{
   \label{eq:10004}
    I_{1}+R_{1,1}=& \frac{1}{2a} \sum_{ i=1 }^{ N }\E  \Big[h(X) D^{[i]} \big\{ F (f(X^{(i)}))-F(f(X))\big\}\Big]\nonumber\\
  =& - \frac{1}{2a} \sum_{ i=1 }^{ N }\E  \left[h(X) D^{[i]} D^{(i)}  F' \left(f(X)-U D^{(i)}\right)\right]\nonumber \\ 
  =& -\frac{1}{2a} \sum_{ i=1 }^{ N }\E  \left[h(X) D^{[i]} D^{(i)} \left\{ F' \left(f(X)-U D^{(i)}\right)- F'(f(X))\right\}\right] \nonumber\\ 
   &+ \frac{1}{a} \E  \left[h(X) F'(f(X)) \left(\sum_{ i=1 }^{ N } \frac{1}{2}\E [(-D^{[i]}) D^{(i)}|X] -b\right) \right]\nonumber\\
   &+  b\E [F'(f(Y))]\nonumber\\
   =:&H_{1}+H_{2}+b\E [F'(f(Y))],
}
where $U$ is a uniform random variable in $[0, 1]$ independent of any other random variables. 
Recalling the definitions of $D^{[i]}$ and $D^{(i)}$, and from \cref{eq:10,eq:11}, we have
\begin{equation}
   \label{eq:10005}
  \begin{aligned}
  H_{2}= O(1) \delta_{2}.
  \end{aligned}
\end{equation}
It remains to consider $H_{1}$. By \cref{eq:10001}, 
\begin{equation}
   \label{eq:10006}
  \begin{aligned}
    H_{1}=&  \frac{1}{2a} \sum_{ i=1 }^{ N }\E  \left[h(X) D^{[i]} D^{(i)} \left\{ 1_{\{f(X)-U D^{(i)}\leq x\}}-1_{\{f(X)\leq x\}} \right\}\right]\\
        &- \frac{1}{2a} \sum_{ i=1 }^{ N }\E  \left[h(X) D^{[i]} D^{(i)} \left\{ (f(X)-U D^{(i)})F(f(X)-U D^{(i)})- f(X)F(f(X))\right\}\right]\\
        =:&H_{1,1}+H_{1,2}.\\
  \end{aligned}
\end{equation}
For $H_{1,1}$, we use the idea from \cite[proof of Theorem~2.1]{shaozhang2019}. Because $1_{\{ w\leq x \}}$ is decreasing w.r.t. $w$, we have 
\begin{equation*}
   \label{eq:10007}
  \begin{aligned}
 0\leq D^{(i)}\left\{ 1_{\{f(X)-U D^{(i)}\leq x\}}-1_{\{f(X)\leq x\}} \right\}\leq D^{(i)}\left\{ 1_{\{f(X^{(i)})\leq x\}}-1_{\{f(X)\leq x\}} \right\}.
  \end{aligned}
\end{equation*}
Thus, recalling the definition of $D^{*}_{i}$ in \cref{eq:Dstar} and using the symmetry \cref{eq:symmetry},  
\ban{
   \label{eq:10008}
    |H_{1,1}|\leq& \frac{1}{2a} \sum_{ i=1 }^{ N }\E  \left[h(X) D^{*}_{i} D^{(i)} \left\{ 1_{\{f(X^{(i)})\leq x\}}-1_{\{f(X)\leq x\}} \right\}\right]\nonumber\\
     =& - \frac{1}{2a} \sum_{ i=1 }^{ N }\E  \left[(h(X)+h(X^{(i)})) D^{*}_{i} D^{(i)}  1_{\{f(X)\leq x\}}\right]\nonumber\\
    =& - \frac{1}{a} \sum_{ i=1 }^{ N }\E  \left[h(X) D^{*}_{i} D^{(i)}  1_{\{f(X)\leq x\}}\right]\nonumber\\
    & + \frac{1}{2a} \sum_{ i=1 }^{ N }\E  \left[(h(X)-h(X^{(i)})) D^{*}_{i} D^{(i)}  1_{\{f(X)\leq x\}} \right]\nonumber\\
    =& O(1) \frac{1}{a} \E  \left[h(X) \left\lvert \sum_{ i=1 }^{ N } \E [D^{*}_{i} D^{(i)} |X] \right\rvert \right]\nonumber\\
     &+ O(1) \frac{1}{a} \sum_{ i=1 }^{ N }\E  \left[h(X) \exp\left\{ \lvert g(X)-g(X^{(i)}) \rvert \right\} D^{*}_{i} |D^{(i)}| \lvert g(X)-g(X^{(i)}) \rvert\right]\nonumber\\
    =& O(1) \delta_{1}'.
}
For $H_{1,2}$, by a similar argument using the boundedness and monotonicity of $wF(w)$ (see, for example, \cite[Lemma~2.3]{chen2010normal}), we have
\begin{equation}
   \label{eq:100010}
  \begin{aligned}
  H_{1,2}= O(1) \delta_{1}'.
  \end{aligned}
\end{equation} 
Thus, by \cref{eq:10008,eq:100010}, we obtain 
\begin{equation}
   \label{eq:100011}
  \begin{aligned}
  H_{1}=O(1) \delta_{1}'.
  \end{aligned}
\end{equation}
Finally, combining \cref{eq:10004,eq:10005,eq:100011}, we complete the proof for \cref{eq:10003}.
\end{proof}

\section{Proof of Theorem \ref{thm:ERGM}}\label{sec:proofERGM}
In this section, we prove \cref{thm:ERGM} using \cref{theorem:1}. 
Let  $C:=C(\beta, H)$ denote positive constants depending only on the parameters $\beta_1,\dots, \beta_k$ and the subgraphs $H_1,\dots, H_k$ in the definition of the ERGM and may differ from line to line. Let $O(1)$ denote constants satisfying $|O(1)|\leq C$. For $j=1,\dots, k$, let $v_j$ and $e_j$ denote the number of vertices and edges of the graph $H_j$, respectively, in the ERGM \cref{ERGM}.

Our proof proceeds as follows. First, in the subcritical region, an ERGM is
close to the Erd\H{o}s--R\'enyi random graph $G(n,p)$ and we write the ERGM as a perturbation of $G(n, p)$ using the Hoeffding decomposition (cf. \cref{eq:perturbation}). This step corresponds to the "centering" step in classical exponential change of measure arguments. Then, we apply the general CLT for nonlinear exponential families developed in \cref{sec:cltexp}.
Finally, we control various error terms in the general CLT for the application to ERGM. In particular, we use higher-order concentration inequalities for weakly dependent random variables.
In Dobrushin's uniqueness region, such higher-order concentration inequalities were developed by \cite{sambale2020logarithmic} using a logarithmic Sobolev inequality (see also an earlier work by \cite{gotze2019higher}). We extend such higher-order concentration inequalities to the whole subcritical region using the approach of \cite{chatterjee2007stein} and \cite{ganguly2019sub}.


We first prove the result in \cref{eq:thmdob} in Dobrushin's uniqueness region where we use \cref{eq:SS20}. In Step 7, we prove the result \cref{eq:thmsub} for the whole subcritical region.

\medskip

\noindent\textbf{Step
1: Centering.}
Rewrite the ERGM \cref{ERGM} as
\be{
p_\beta(G)\propto \exp\{\sum_{j=1}^k [\frac{\beta_j}{n^{v_j-2}}|\text{Hom}(H_j, G)|-2\beta_j e_j p^{e_j-1}E]\} \exp (2\sum_{j=1}^k \beta_j e_j p^{e_j-1}E),
}
where $E$ denotes the number of edges in $G$.
Because $p$ satisfies \cref{eq:p}, we have
\ben{\label{eq:perturbation}
p_\beta(G)\propto \exp\{\sum_{j=1}^k [\frac{\beta_j}{n^{v_j-2}}|\text{Hom}(H_j, G)|-2\beta_j e_j p^{e_j-1}E]\} p^E (1-p)^{N-E}.
}
The motivation for this rewriting is that in order to have a smaller variance, we have subtracted inside the brackets $[\cdots]$ the leading term in the Hoeffding decomposition (\cite{hoeffding1948}) under $G(n,p)$. This way, we can view the ERGM \cref{ERGM} as a perturbation of $G(n,p)$. This step corresponds to the ``centering" step in classical exponential change of measure arguments. See \cite{ding2024second} for another application of this observation. Technically, such a centering is useful in bounding, for example, $\delta_1$ in \cref{eq:del1.2}.

\medskip

\noindent\textbf{Step 2: Formulation of a nonlinear exponential family.}
Recall that we identify a simple graph on $n$ labeled vertices $\{1,\dots, n\}$ with an element $x=(x_{ij})_{1\leq i<j\leq n}\in \{0,1\}^{\mcl{I}}$, where ${\mcl{I}}:={\mcl{I}_n}:=\{s=(i,j):1\leq i<j\leq n\}$ and $x_{ij}=1$ if and only if there is an edge between vertices $i$ and $j$. In this way, the ERGM \eq{ERGM} induces a random element $Y\in \{0,1\}^{\mcl{I}}$. Similarly, $G(n,p)$ induces a random element $X$ in $\{0,1\}^{\mcl{I}}$. 
This puts the problem into the framework of nonlinear exponential families considered in \cref{theorem:1}, with
$X=\{X_{s}: s\in \mcl{I}\}$ being i.i.d.\ $Bernoulli(p)$, $N={n \choose 2}$, and $Y=\{Y_{s}: s\in \mcl{I}\}$ following \cref{eq:1} with
\ben{\label{eq:hg}
h(x)=\exp(g(x))=\exp\{\sum_{j=1}^k [\frac{\beta_j}{n^{v_j-2}}|\text{Hom}(H_j, G)|-2\beta_j e_j p^{e_j-1}E]\}.
}
Let 
\be{
f(x)=\frac{\sum_{1\leq i<j\leq n}x_{ij}-\mu_n}{\sigma_n},
}
and let 
\ben{\label{eq:centeredsum}
W=W_n=f(Y)=\frac{\sum_{1\leq i<j\leq n} Y_{ij}-\mu_n}{\sigma_n},
}
where $\mu_n=\E \sum_{1\leq i<j\leq n} Y_{ij}$, and $\sigma_n\asymp n$ as in \cref{eq:sigmasq}.

\medskip

\noindent\textbf{Step 3: Computing $b$, $\Delta_{1,s}$ and $\Delta_{2,s}$.}
We change the index $i\in \{1,\dots, N\}$ in \cref{theorem:1} to $s\in \mcl{I}$ in this proof because $i$ is used as vertex index.
We can compute (from \cref{eq:t4}, \cref{eq:centeredsum}, $X_s^2=X_s$ and $X_s'$ is an independent copy of $X_s$)
\ben{\label{eq:Delta1.1}
\Delta_{1,s}(X)=\frac{1}{2\sigma_n^2}\E[(X_s-X_s')^2\vert X]=\frac{1}{2\sigma_n^2} [(1-2p)X_s+p].
}
From \cref{eq:b} and \cref{eq:LLN}, we have
\ben{\label{eq:basym}
b=\frac{1}{2\sigma_n^2} \sum_{s\in \mcl{I}}\E[(1-2p)Y_s+p]=\frac{N}{2\sigma_n^2} 2p(1-p)+O(\frac{1}{\sqrt{n}}).
}
From \cref{eq:t5,eq:centeredsum,eq:hg}, we have
\besn{\label{eq:Delta2.1}
&\sum_{s\in \mcl{I}} \Delta_{2,s}(X)\\
=&\sum_{s\in \mcl{I}} \sum_{j=2}^k \frac{\beta_j}{2n^{v_j-2}} \E\left\{\Big[|\text{Hom}(H_j, G_X)|-|\text{Hom}(H_j, G_X^{(s)})|-2n^{v_j-2} e_j p^{e_j-1} (X_s-X_s')\Big]\frac{(X_s-X_s')}{\sigma_n} \vert X \right\},
}
where $|\text{Hom}(H_j, G_X)|$ denotes the number of homomorphisms of $H$ into the random graph $G_X$ (corresponding to the edge indicators vector $X$) and the random graph $G_X^{(s)}$ differs from $G_X$ only by replacing the edge indicator $X_s$ with the independent copy $X_s'$.
Note that
\ben{\label{eq:diffhom}
|\text{Hom}(H_j, G_X)|-|\text{Hom}(H_j, G_X^{(s)})|=(X_s-X_s') |\text{Hom}(H_j, G_X, s)|,
}
where $|\text{Hom}(H_j,G_X, s)|$ denotes the number of homomorphisms of $H_j$ into $G_X$ but requiring that an edge of $H_j$ must be mapped to the edge $s$ (no matter the edge $s$ is present in $G_X$ or not).
From \cref{eq:Delta2.1,eq:diffhom} and a computation of conditional expectation as in \cref{eq:Delta1.1}, we obtain
\ben{\label{eq:del2.1}
\sum_{s\in \mcl{I}} \Delta_{2,i}(X)=\sum_{j=2}^k \frac{\beta_j}{2n^{v_j-2}\sigma_n}\sum_{s\in \mcl{I}} [(1-2p)X_{s}+p][|\text{Hom}(H_j, G_X, s)|-2n^{v_j-2}e_j p^{e_j-1}].
}
From the computation of $b$ in \cref{eq:basym} and the fact that $N\asymp \sigma_n^2$, to prove \cref{eq:thmdob} in \cref{thm:ERGM} using \cref{theorem:1}, it suffices to show $\delta_1$, $\delta_1'$, $\delta_2$ and $\delta_3$ in \cref{eq:t1,eq:t1',eq:t2,eq:t3} are all bounded by $C/\sqrt{n}$.

\medskip

\noindent\textbf{Step 4: Higher-order concentration inequalities.}
To bound error terms appearing in \cref{theorem:1} for the application to ERGM in Dobrushin's uniqueness region, we rely on the results of \cite{sambale2020logarithmic} (see also the earlier result of \cite{gotze2019higher}) on higher-order concentration inequalities. We introduce the results of \cite{sambale2020logarithmic} for our application in this step.

As in  \cite{sambale2020logarithmic}, as building blocks of the Hoeffding decomposition under the ERGM $p_\beta$ in Dobrushin's uniqueness region, we define centered random variables\footnote{\cite{sambale2020logarithmic} had $1_{\{N(P)=0\}}$ instead of $M(P) 1_{\{N(P)=0\}}$ in their definition of $f_{d, A}$. However, for $f_{d, A}$ to be centered as claimed in the first paragraph of their proof of Theorem 3.7, we need to account for multiplicity. The other parts of their proof go through with this new definition of $f_{d, A}$.}
\besn{\label{eq:hoeffp2}
    &f_{d,A}(Y):=\sum_{I\in \mcl{I}^d} A_I g_I(Y)\\
    :=&\sum_{I\in \mcl{I}^d} A_I \sum_{P\in \mcl{P}(I)}(-1)^{M(P)} \left\{ M(P)1_{\{N(P)=0\}}+\prod_{J\in P\atop |J|=1}  (Y_J-\t p)1_{\{N(P)>0\}}\right\} \prod_{J\in P\atop |J|>1} \left\{ \E \prod_{l\in J} (Y_l-\t p) \right\},
}
where $d\geq 1$ is an integer, $\mcl{I}=\{(i,j): 1\leq i<j\leq n\}$, $A$ is a $d$-tensor with vanishing diagonal (i.e., $A_I\ne 0$ only if all the $d$ elements in $I$ are distinct),
\begin{equation*}
    \mcl{P}(I)=\{S\subset 2^I: S\ \text{is a partition of}\ I\},
\end{equation*}
\be{
\t p:=\E Y_l,
}
$N(P)$ ($M(P)$, resp.) is the number of subsets with one element (more than one element, resp.) in the partition $P$ and for a singleton set $J=\{l\}$, $Y_J:=Y_l$. 
We will write $f_{d, A}:=f_{d, A}(Y)$ and $g_I:=g_I(Y)$ for simplicity of notation. 
For our purpose, it is enough to consider those values of $d$ bounded by the maximum number of edges of graphs $H_1,\dots, H_k$.
From \cite[Theorem~3.7]{sambale2020logarithmic},\footnote{They assumed that the graphs $H_1,\dots, H_k$ in the definition of ERGM are all connected at the beginning of their paper. However, as far as we checked their proof, this requirement is not needed.} in Dobrushin's uniqueness region \cref{eq:Dobrushin},
\begin{equation}\label{eq:SS20}
   \|f_{d,A}\|_p\leq C_p \|A\|_2, 
\end{equation}
where $\|\cdot\|_p$, $p\geq 1$, denotes the $L^p$-norm of a random variable, $C_p$ is a constant depending on $p$ in addition to the parameters in the ERGM, and $\|A\|_2$ is the Euclidean norm of the tensor $A$ when viewed as a vector.

Recall $\t p:=\E Y_l$ (same for all $l$ by symmetry) and let $\t Y_l:= Y_l-\t p$. For $m$ distinct indices $s_1,\dots, s_m\in \mcl{I}$, by considering $Y_{s_i}=\t Y_{s_i}+\t p$, we can write
\ben{\label{eq:decomp1}
Y_{s_1}\cdots Y_{s_m}-\E [Y_{s_1}\cdots Y_{s_m}]
=\sum_{l=1}^m \t p^{m-l} \sum_{1\leq i_1<\dots <i_l\leq m} \t Y_{s_{i_1}} \cdots \t Y_{s_{i_l}} -\text{mean},
}
where $\text{mean}$ denotes the expectation of the random variable in front of it. 

To relate \cref{eq:decomp1} to \cref{eq:hoeffp2}, we need the following result. From \cite[Eq.(34)]{ganguly2019sub}, for any fixed $m\geq 1$ and distinct edges $l_1,\dots, l_m$, we have, in the subcritical region,
\begin{equation}\label{eq:GN19}
    |\E(Y_{l_1} \vert Y_{l_2},\dots, Y_{l_m})-\E Y_{l_1}|\leq \frac{C}{n}.
\end{equation}
Completing each $\t Y_{s_{i_1}} \cdots \t Y_{s_{i_l}}$ to $g_{\{s_{i_1}, \dots, s_{i_l}\}}$ (recall \cref{eq:hoeffp2}) in \cref{eq:decomp1} and using \cref{eq:GN19}, we obtain
\be{
Y_{s_1}\cdots Y_{s_m}-\E [Y_{s_1}\cdots Y_{s_m}]= \sum_{l=1}^m \left(\t p^{m-l}+O(\frac{1}{n})\right)\sum_{1\leq i_1<\cdots<i_l\leq m} g_{\{s_{i_1},\dots, s_{i_l}\}}.
}
From $\t p=p+O(1/\sqrt{n})$ (cf. \cref{eq:LLN}), we obtain
\ben{\label{eq:decomp2}
Y_{s_1}\cdots Y_{s_m}-\E[ Y_{s_1}\cdots Y_{s_m}]= \sum_{l=1}^m \left(p^{m-l}+O(\frac{1}{\sqrt{n}})\right)\sum_{1\leq i_1<\cdots<i_l\leq m} g_{\{s_{i_1},\dots, s_{i_l}\}}.
}
At a high level, \cref{eq:decomp2} implies (see details in the next step) that
similar to the Hoeffding decomposition for functionals of independent random variables, centered graph (say, $H$) counting statistics can be written as a linear combination of $f_{d, A}$ in \cref{eq:hoeffp2} for $1\leq d\leq |\mathcal{E}(H)|$, where $\mathcal{E}(\cdot)$ denotes the edge set. See, for example, \cite[Eq.(28)]{sambale2020logarithmic} for triangle counts. 
Each term in the decomposition is indexed by a subgraph $H_i$ of $H$. Moreover, each term  equals $c_n f_{d,A}$, with $d=|\mathcal{E}(H_i)|$, $A$ having entries of constant order and $c_n\asymp n^{|\mathcal{V}(H)|-|\mathcal{V}(H_i)|}$. It can be checked from \cref{eq:SS20} that the leading term is indexed by edges (this is the same as in the classical Hoeffding decomposition).

\medskip

\noindent\textbf{Step 5: Bounding $\delta_1, \delta_2$ and $\delta_3$.}
For $\delta_1$ in \cref{eq:t1}, using $|g(X)-g(X^{(s)}|\leq C$ (recall \cref{eq:hg}) and $|f(X)-f(X^{(s)})|\leq C/n$, we have
\ben{\label{eq:del1.1}
\delta_1\leq \frac{C}{na} \sum_{s\in \mcl{I}}\E\left[ h(X) |g(X)-g(X^{(s)})|^3\right]+\frac{C}{n}.
}
From the expression of $g(x)$ in \cref{eq:hg}, we have
\ben{\label{eq:gXgXs}
|g(X)-g(X^{(s)})|\leq \sum_{j=2}^k \frac{\beta_j}{n^{v_j-2}} ||\text{Hom}(H_j, G_X, s)|-2n^{v_j-2} e_j p^{e_j-1}|.
}
From \cref{eq:transfer}, for any $j=2,\dots, k$ and $s\in \mcl{I}$, we have
\besn{\label{eq:del1.2}
&\frac{n}{a} \E h(X)\left|\frac{|\text{Hom}(H_j, G_X, s)|-2n^{v_j-2} e_j p^{e_j-1}}{n^{v_j-2}} \right|^3\\
=& n \E \left|\frac{|\text{Hom}(H_j, G_Y, s)|-2n^{v_j-2} e_j p^{e_j-1}}{n^{v_j-2}} \right|^3,
}
where $G_Y$ is the random graph with edge indicators vector $Y$.
Note that
\be{
|\text{Hom}(H_j, G_Y, s)|-\text{mean}=\sum_{\{r_1,\dots, r_{e_j-1}\}\subset \mcl{I}:\atop r_1\cup \cdots \cup r_{e_j-1}\cup s \cong H_j} |\text{Aut}(H_j)| Y_{r_1}\cdots Y_{r_{e_j-1}}-\text{mean},
}
where $\cong$ denotes graph isomorphism, $|\text{Aut}(H_j)|$ denotes the number of automorphisms of $H_j$ to itself, and $-\text{mean}$ always denote centering of relevant random variables.
From \cref{eq:decomp2}, we have
\be{
Y_{r_1}\cdots Y_{r_{e_j-1}} - \E [Y_{r_1}\cdots Y_{r_{e_j-1}}] =\sum_{A\subset \{r_1,\dots, r_{e_j-1}\}:\atop |A|\geq 1} c_{A, \{r_1,\dots, r_{e_j-1}\}} g_A,
}
where $c_{A, \{r_1,\dots, r_{e_j-1}\}}$ are constants uniformly bounded by $C$.
Therefore, we can further write $|\text{Hom}(H_j, G_Y, s)|-\text{mean}$ as
\be{
\sum_{v=3}^{v_j}\sum_{d=1}^{e_j-1}\sum_{A\subset \mcl{I}:\atop |A|=d, |\mathcal{V}(A\cup s)|=v} \left(\sum_{r_1,\dots, r_{e_j-1}\subset \mcl{I}: \atop r_1\cup  \cdots \cup r_{e_j-1} \cup s\cong H_j, \{r_1,\dots, r_{e_j-1}\} \supset A} |\text{Aut}(H_j)| c_{A, \{r_1,\dots, r_{e_j-1}\}} \right) g_A.
}
Observe that for fixed $v=3,\dots, v_j$ and $d=1,\dots, e_j-1$, the coefficient in front of the $g_A$'s inside of the brackets $(\cdot)$ above has $O(n^{v-2})$ non-zero entries (the number of choices of the other $v-2$ vertices of $A\cup s$ except for those 2 vertices connecting $s$), each entry is of the order $O(n^{v_j-v})$ (the number of the remaining $v_j-v$ vertices of $r_1\cup \dots \cup r_{e_j-1}\cup s$). From \cref{eq:SS20}, the order of the random variable $|\text{Hom}(H_j, G(X), s)|-\text{mean}$ is
\ben{\label{eq:del1.4}
O(1) \sum_{v=3}^{v_j} \sum_{d=1}^{e_j-1}\sqrt{n^{v-2} n^{2(v_j-v)}}=O(1)\sum_{v=3}^{v_j} n^{v_j-\frac{v}{2}-1} =O(n^{v_j-\frac{5}{2}}).
}
From \cref{eq:GN19} and \cref{eq:LLN}, the expectation of the normalized $|\text{Hom}|$ is close to the subtracted term in \cref{eq:del1.2} with error rate $O(1/\sqrt{n})$.
This, together with \cref{eq:del1.4}, implies that in the Dobrushin's uniqueness region
\ben{\label{eq:del1.3}
n \E \left|\frac{|\text{Hom}(H_j, G_Y, s)|-2n^{v_j-2} e_j p^{e_j-1}}{n^{v_j-2}} \right|^3 \leq \frac{C}{\sqrt{n}}.
}
See \cref{lemma:1004} for a different proof of \cref{eq:del1.3} that works for the whole subcritical region.

The bounds \cref{eq:del1.1,eq:gXgXs,eq:del1.2,eq:del1.3} imply
\be{
\delta_1\leq \frac{C}{\sqrt{n}}.
}

For $\delta_2$ in \cref{eq:t2}, from \cref{eq:Delta1.1}, we have
\begin{equation*}
  \begin{aligned}
  \Var\left(\sum_{ s\in \mathcal{I} }^{ } \Delta_{1,s}(Y)\right)= \frac{(1-2p)^{2}}{4 \sigma_{n}^{4}} \Var\left(\sum_{ s\in \mathcal{I} }^{  } Y_{s}\right).
  \end{aligned}
\end{equation*}
From \cite[Theorem~1]{ganguly2019sub}), we have 
\ben{\label{eq:gangulyvar}
\Var(\sum_{s\in \mcl{I}}Y_s)\leq C n^2,
}
and we will use this fact several times in the remaining proof.
Together with the fact that $\sigma_{n}^{2}\asymp n^2$, we obtain
\ben{\label{0del2}
\delta_2=\sqrt{\Var\left(\sum_{ s\in \mathcal{I} }^{ } \Delta_{1,s}(Y)\right) } \leq \frac{C}{n}.
}

For $\delta_3$ in \cref{eq:t3}, from \cref{eq:del2.1} and \cref{eq:basym}, we have
\bes{
\Delta_3(Y):=&\sum_{ s\in \mcl{I} } \left(\Delta_{2,i}(Y) -  \E  \Delta_{2,i}(Y)\right)-(1-b) f(Y)\\
=& \sum_{j=2}^k \frac{\beta_j}{2n^{v_j-2}\sigma_n} \left\{\sum_{s\in \mcl{I}} \Big[(1-2p)Y_s+p\Big]\Big[|\text{Hom}(H_j, G_Y, s)|-2n^{v_j-2} e_j p^{e_j-1}\Big] -\text{mean}\right\}\\
&-\left[(1-\frac{Np(1-p)}{\sigma_n^2})+O(\frac{1}{\sqrt{n}})\right]f(Y),
}
where $-\text{mean}$ denotes the centering of the random variable inside the brackets $\{\cdots\}$.
Expanding the product $[\cdots][\cdots]$ in the above equation, we obtain
\besn{\label{eq:Delta3}
\Delta_3(Y)=&\sum_{j=2}^k \frac{\beta_j(1-2p)}{2n^{v_j-2}\sigma_n}\left\{ \sum_{s\in \mcl{I}} Y_{s}|\text{Hom}(H_j,G_Y, s)|-2n^{v_j-2} e_j p^{e_j-1}\sigma_n f(Y)-\text{mean}\right\}\\
&+\sum_{j=2}^k \frac{\beta_j p}{2n^{v_j-2}\sigma_n} \left\{\sum_{s\in \mcl{I}} |\text{Hom}(H_j,G_Y, s)|-\text{mean}\right\}\\
&-\left[(1-\frac{Np(1-p)}{\sigma_n^2})+O(\frac{1}{\sqrt{n}})\right]f(Y).
}
For the first term, we first show that
\be{
\sum_{s\in \mcl{I}} Y_{s}|\text{Hom}(H_j, G_Y, s)|=e_j |\text{Hom}(H_j, G_Y)|.
}
In fact, if we label the vertices of $H_{j}$ by $l_1,l_2,\ldots l_{v_j}$, and
 let $[n]:= \{ 1,2, \ldots ,n \}$, then, 
\begin{equation}
   \label{eq:1001}
  \begin{aligned}
    |\text{Hom}(H_{j},G_Y)|= \sum_{\begin{subarray}{c} k_{1},k_{2}, \ldots ,k_{v_{j}}\in [n] \\ k_{1},k_{2}, \ldots ,k_{v_{j}} \text{ are distinct}\end{subarray}} \prod_{\{ l_p,l_q \}\in \mathcal{E}(H_{j}) } Y_{k_{p},k_{q}}.
  \end{aligned}
\end{equation}
Furthermore, we have, suppose the edge $s$ connects the two vertices $1\leq s_1<s_2\leq n$,
\ba{
  & \sum_{1\leq s_1<s_2\leq n} Y_{s_1,s_2}|\text{Hom}(H_j,G_Y, \{s_1,s_2\})|=\frac{1}{2}\sum_{1\leq s_1\neq s_2\leq n} Y_{s_1,s_2}|\text{Hom}(H_j,G_Y, \{s_1,s_2\})|\\
    =&  \frac{1}{2}\sum_{1\leq s_1\neq s_2\leq n} Y_{s_1,s_2} \sum_{\begin{subarray}{c} k_{1},k_{2}, \ldots ,k_{v_{j}}\in [n] \\ k_{1},k_{2}, \ldots ,k_{v_{j}} \text{ are distinct}\end{subarray}} \sum_{ 1\leq p\neq q\leq v_{j} } 1_{\{ s_1=k_{p},s_2=k_{q} \}}1_{\{ \{ l_p,l_q \}\in \mathcal{E}(H_{j}) \}}   \prod_{\{ l_u,l_v \}\in \mathcal{E}(H_{j})\backslash \{ l_p,l_q \} } Y_{k_{u},k_{v}}\\
    =& \frac{1}{2}\sum_{ 1\leq p\neq q\leq v_{j} }1_{\{ \{ l_p,l_q \}\in \mathcal{E}(H_{j}) \}}\sum_{\begin{subarray}{c} k_{1},k_{2}, \ldots ,k_{v_{j}}\in [n] \\ k_{1},k_{2}, \ldots ,k_{v_{j}} \text{ are distinct}\end{subarray}} \prod_{\{ l_u,l_v \}\in \mathcal{E}(H_{j}) } Y_{k_{u},k_{v}}\\
    =& e_{j} |\text{Hom}(H_{j},G_Y)|, 
}
where $1_{\{\cdot\}}$ is the indicator function. The leading term (indexed by edges) in the Hoeffding decomposition of $[e_j |\text{Hom}(H_j, G_Y)|-\text{mean}]$ computed from \cref{eq:decomp2} is $2n^{v_j-2} e_j^2 p^{e_j-1} \sigma_n f(Y)$.

Similarly, for the second term in \cref{eq:Delta3},
\ban{
  &\sum_{1\leq s_1<s_2\leq n} |\text{Hom}(H_j,G_Y, \{s_1,s_2\})|
 =\frac{1}{2}\sum_{1\leq s_1\neq s_2\leq n} |\text{Hom}(H_j,G_Y, \{s_1,s_2\})|\nonumber \\
 =&  \frac{1}{2}\sum_{1\leq s_1\neq s_2\leq n}  \sum_{\begin{subarray}{c} k_{1},k_{2}, \ldots ,k_{v_{j}}\in [n] \\ k_{1},k_{2}, \ldots ,k_{v_{j}} \text{ are distinct}\end{subarray}} \sum_{ 1\leq p\neq q\leq v_{j} } 1_{\{ s_1=k_{p},s_2=k_{q} \}}1_{\{ \{ l_p,l_q \}\in \mathcal{E}(H_{j}) \}}   \prod_{\{ l_u,l_v \}\in \mathcal{E}(H_{j})\backslash \{ l_p,l_q \} } Y_{k_{u},k_{v}}\nonumber \\
 =&  \frac{1}{2}  \sum_{\begin{subarray}{c} k_{1},k_{2}, \ldots ,k_{v_{j}}\in [n] \\ k_{1},k_{2}, \ldots ,k_{v_{j}} \text{ are distinct}\end{subarray}} \sum_{ 1\leq p\neq q\leq v_{j} } 1_{\{ \{ l_p,l_q \}\in \mathcal{E}(H_{j}) \}}   \prod_{\{ l_u,l_v \}\in \mathcal{E}(H_{j})\backslash \{ l_p,l_q \} } Y_{k_{u},k_{v}}\nonumber \\
 =&  \frac{1}{2}  \sum_{ 1\leq p\neq q\leq v_{j} } 1_{\{ \{ l_p,l_q \}\in \mathcal{E}(H_{j}) \}}\sum_{\begin{subarray}{c} k_{1},k_{2}, \ldots ,k_{v_{j}}\in [n] \\ k_{1},k_{2}, \ldots ,k_{v_{j}} \text{ are distinct}\end{subarray}}    \prod_{\{ l_u,l_v \}\in \mathcal{E}(H_{j})\backslash \{ l_p,l_q \} } Y_{k_{u},k_{v}}\nonumber \\
=&  \frac{1}{2}  \sum_{ 1\leq p\neq q\leq v_{j} } 1_{\{ \{ l_p,l_q \}\in \mathcal{E}(H_{j}) \}}
 |\text{Hom}(H_{j}\backslash\{l_p,l_q\},G_Y)|,  \label{eq:1003}
 }
where $H_{j}\backslash \{ l_p,l_q \}$ denotes the subgraph of $H_{j}$ by deleting the edge $\{ l_p,l_q \}$ in $H_{j}$ (but keeping all the vertices even they become isolated). Since the leading term (indexed by edges) in the Hoeffding decomposition of $|\text{Hom}(H_{j}\backslash\{l_p,l_q\},G_Y)|$ computed from \cref{eq:decomp2} is 
\begin{equation*}
  \begin{aligned}
 2 n^{v_{j}-2} (e_{j}-1) p^{e_{j}-2} \sigma_n f(Y),
  \end{aligned}
\end{equation*}
we subtract 
\begin{equation*}
  \begin{aligned}
    \frac{1}{2}  \sum_{ 1\leq p\neq q\leq v_{j} } 1_{\{ \{ l_p,l_q \}\in \mathcal{E}(H_{j}) \}}  2 n^{v_{j}-2} (e_{j}-1) p^{e_{j}-2} \sigma_n f(Y)= 2 n^{v_{j}-2} e_{j} (e_{j}-1) p^{e_{j}-2} \sigma_n f(Y)
  \end{aligned}
\end{equation*}
from $\sum_{1\leq s_1<s_2\leq n} |\text{Hom}(H_j, G_Y, \{s_1,s_2\})|$ in the equality below to obtain
\begin{equation}
   \label{eq:1006}
  \begin{aligned}
 &\Delta_3(Y)\\=&\sum_{j=2}^k \frac{\beta_j(1-2p)}{2n^{v_j-2}\sigma_n}\left\{  e_j |\text{Hom}(H_j, G_Y)|-2n^{v_j-2} e_j^2 p^{e_j-1} \sigma_n f(Y)-\text{mean}\right\}\\
            &+\sum_{j=2}^k \frac{\beta_j p}{2n^{v_j-2}\sigma_n} \left\{\frac{1}{2}  \sum_{\begin{subarray}{c}
1\leq p\neq q\leq v_{j}\\\{ l_p,l_q \}\in \mathcal{E}(H_{j})
\end{subarray}} 
\left[|\text{Hom}(H_{j}\backslash\{l_p,l_q\},G_Y)|-2 n^{v_{j}-2} (e_{j}-1) p^{e_{j}-2} \sigma_n f(Y)\right]-\text{mean}\right\}\\
&+\left\{\sum_{j=2}^k \frac{\beta_j(1-2p)}{2}(2e_j^2 p^{e_j-1}-2e_j p^{e_j-1})+\sum_{j=2}^k \frac{\beta_j p}{2} 2e_j(e_j-1)p^{e_j-2}\right\} f(Y)\\
&-\left[(1-\frac{p(1-p)}{\sigma_n^2/N})+O(\frac{1}{\sqrt{n}})\right] f(Y), 
  \end{aligned}
\end{equation}
where the coefficients were chosen above so that the leading terms indexed by edges in the Hoeffding decomposition of the $|\text{Hom}|$'s were subtracted (in order to apply \cref{lemma:1001} below).

Recall $\delta_3=\sqrt{\Var(\Delta_3(Y))}$.
From the choice of $\sigma_n^2$ in \cref{eq:sigmasq} and $\E [f^2(Y)]\leq C$ (recall \cref{eq:gangulyvar}), the standard deviation of the difference of the last two terms in \cref{eq:1006} is of the order $O(1/\sqrt{n})$.

To complete the proof, we need the following lemma, which is a straightforward generalization of the result of \cite{sambale2020logarithmic} on triangle counts.  
\begin{lemma}\label{lemma:1001}
For any fixed graph $H$, which can have isolated vertices, with
$e=|\mathcal{E}(H)|\geq 1$ and $v=|\mathcal{V}(H)|$, then in the  Dobrushin's 
 uniqueness region,  we have 
\begin{equation}
   \label{eq:1007}
  \begin{aligned}
    \Var\left(|\text{Hom}(H,G_Y)|- 2 n^{v-2} e p^{e-1} \sigma_{n}f(Y)\right) = O(n^{2v-3}). 
  \end{aligned}
\end{equation}

\end{lemma}
\begin{proof}[Proof of \cref{lemma:1001}]
It suffices to consider the case that $H$ does not have isolated vertices. Otherwise, each isolated vertex contributed to a factor of $n^2$ to both sides of the equation \cref{eq:1007}.

We label the vertices of $H$ by $l_1,l_2,\dots, l_v$. We also think of each homomorphism of $H$ into $G_Y$ as a map of $v$ distinct vertices $k_1,\dots, k_v\in [n]$ to the vertices $l_1,\dots, l_v$ of $H$ such that $Y_{k_i, k_j}=1$ if $\{l_i,l_j\}$ is an edge in $H$.

To apply \cref{eq:SS20}, we first rewrite $|\text{Hom}(H,G_Y)|$ as a linear combination of $g$ in \cref{eq:hoeffp2}.
From \cref{eq:1001} and \cref{eq:decomp2}, we have 
\ba{
   & |\text{Hom}(H,G_Y)|-\text{mean}\\=& \sum_{\begin{subarray}{c} k_{1},k_{2}, \ldots ,k_{v}\in [n] \\ k_{1},k_{2}, \ldots ,k_{v} \text{ are distinct}\end{subarray}} \prod_{\{ l_p,l_q \}\in \mathcal{E}(H) } Y_{k_{p},k_{q}}-\text{mean}\\
 =&  \sum_{\begin{subarray}{c} k_{1},k_{2}, \ldots ,k_{v}\in [n] \\ k_{1},k_{2}, \ldots ,k_{v} \text{ are distinct}\end{subarray}} \sum_{ d=1 }^{ e } \sum_{ r=2 }^{ v }\sum_{
   \begin{subarray}{c}
A\subset H, \lvert \mathcal{E}(A) \rvert=d,\lvert \mathcal{V}(A) \rvert=r,\\ A \text{ 
 do not have isolate vertices}\\
\end{subarray}
  }(p^{e-d}+O(\frac{1}{\sqrt{n}})) g_{ \{\{k_p,k_q\}:1\leq p<q\leq v,  \{l_p,l_q\}\in \mathcal{E}(A)\}} \\
 =&\sum_{ d=1 }^{ e } \sum_{ r=2 }^{ v } \sum_{
   \begin{subarray}{c}
A\subset H, \lvert \mathcal{E}(A) \rvert=d,\lvert \mathcal{V}(A) \rvert=r,\\ A \text{ 
 do not have isolate vertices}
\end{subarray}
  }
  \sum_{ \begin{subarray}{c} k_{1},k_{2}, \ldots ,k_{r}\in [n] \\ k_{1},k_{2}, \ldots ,k_{r} \text{ are distinct} 
        \end{subarray} } \left(\prod_{m=1}^{v-r} (n-r-m+1)\right)\\
  &\qquad\qquad\qquad\qquad \times\left(p^{e-d}+O(\frac{1}{\sqrt{n}})\right) g_{ \{\{k_p,k_q\}: 1\leq p<q\leq v, \{l_p,l_q\}\in \mathcal{E}(A)\}},
}
where the summations over $A$ are over all the labeled subgraphs of $H$.
For example, if $H$ is a rectangle with
$$\mathcal{V}(H)=\{l_1,l_2,l_3,l_4\}, \text{ and } \mathcal{E}(H)=\{\{l_1,l_2\}, \{l_2,l_3\}, \{l_3,l_4\}, \{l_4,l_1\}\}$$
then for $d=3$, $A$ can be $A_1,A_2,A_3,A_4$ with
$$\mathcal{V}(A_1)=\{l_1,l_2,l_3,l_4\}, \text{ and } \mathcal{E}(A_1)=\{\{l_1,l_2\}, \{l_2,l_3\}, \{l_3,l_4\}\};$$
$$\mathcal{V}(A_2)=\{l_1,l_2,l_3,l_4\}, \text{ and } \mathcal{E}(A_1)=\{\{l_1,l_2\}, \{l_2,l_3\}, \{l_4,l_1\}\};$$
$$\mathcal{V}(A_3)=\{l_1,l_2,l_3,l_4\}, \text{ and } \mathcal{E}(A_1)=\{\{l_1,l_2\},  \{l_3,l_4\}, (l_4,l_1\}\};$$
$$\mathcal{V}(A_4)=\{l_1,l_2,l_3,l_4\}, \text{ and } \mathcal{E}(A_1)=\{\{l_2,l_3\}, \{l_3,l_4\}, (l_4,l_1\}\}.$$
The leading term (corresponding to $d=1$ and $r=2$) is equal to 
\begin{equation*}
   \label{eq:1009}
  \begin{aligned}
    2 e \left(p^{e-1}+O(\frac{1}{\sqrt{n}})\right) \prod_{m=2}^{v-1} (n-m)  \sigma_{n} f(Y).
  \end{aligned}
\end{equation*}
From \cref{eq:SS20} and a similar discussion leading to \cref{eq:del1.4}, we have 
\begin{equation}
   \label{eq:1010}
  \begin{aligned}
  \Var\left(|\text{Hom}(H,G_Y)|-2 e \left(p^{e-1}+O(\frac{1}{\sqrt{n}})\right) \left(\prod_{m=2}^{v-1} (n-m)\right)  \sigma_{n} f(Y)\right)=O(n^{2(v-3)}n^{3})=O(n^{2v-3}).
  \end{aligned}
\end{equation}
By \cref{eq:1010} and using the facts that $\Var(f(Y))=O(1)$ and $\left(\prod_{m=2}^{v-1} (n-m)\right)-n^{v-2}=O(n^{v-3})$, we complete the proof.
\end{proof}

Now, applying \cref{lemma:1001} to each "$|$Hom$(\cdot)|-$leading term" in the first two terms of \cref{eq:1006}, we obtain
\be{
\delta_3=\sqrt{\Var(\Delta_3(Y))}\leq \frac{C}{\sqrt{n}}.
}
This finishes the proof of \cref{eq:thmdob} for the Wasserstein distance in \cref{thm:ERGM}. 

\medskip

\noindent\textbf{Step 6: Bounding $\delta_1'$.}
Finally, we bound $\delta_1'$ and obtain the Kolmogorov bound in \cref{eq:thmdob}. 
Since all the edge indicators are bounded by $1$, we can take $D^{*}_{i}(X, X')= 1/\sigma_{n}$. Similar to the proof for $\delta_1$ in \cref{eq:del1.1}, we have 
\begin{equation}
   \label{eq:100012}
  \begin{aligned}
 & \frac{1}{a}\sum_{ s\in \mcl{I} } \E\Big\{ h(X) \exp\left[ |g(X)- g(X^{(s)})| \right] D^*_{i}(X,X') \left\lvert f(X)- f(X^{(s)})\right\rvert |g(X)- g(X^{(s)})|\Big\}\\
    \leq &  \frac{C}{n^2 a} \sum_{ s\in \mathcal{I} }^{  } \E \big[h(X) \lvert g(X)-g(X^{(s)}) \rvert\big]\\ 
    \leq & \frac{C}{\sqrt{n}}, \\
  \end{aligned}
\end{equation}
where the last inequality follows from a similar argument as that leading to \cref{eq:del1.3}. For the second term in $\delta_{1}'$, with $D^{*}_{i}(X, X')= 1/\sigma_{n}$ and by \cref{eq:LLN}, \cref{eq:transfer} and the facts that $\sigma_n^2\asymp n^2$ and $\Var(\sum_{s\in \mcl{I}} Y_s)\leq C n^2$ (recall \cref{eq:gangulyvar}), we have
\begin{equation}
   \label{eq:100013}
  \begin{aligned}
&\frac{1}{a}\E\left\{ h(X) \left\lvert\sum_{ i=1 }^{ N } \E [ D^*_{i}(X,X')(f(X)- f(X^{(i)}))|X]\right\rvert\right\}\\
=&\frac{1}{a \sigma_{n}^{2}}\E\left\{ h(X) \left\lvert\sum_{ s\in \mathcal{I} }^{ } (X_{s}-p) \right\rvert\right\}
=  \frac{1}{\sigma_n^{2}} \E  \left\lvert\sum_{ s\in \mathcal{I} }^{ } (Y_{s} -p) \right\rvert 
\leq  \frac{C}{\sqrt{n}}.
  \end{aligned}
\end{equation}
Combining \cref{eq:100012,eq:100013}, we obtain
\begin{equation*}
   \label{eq:100014}
  \begin{aligned}
  \delta_{1}' \leq  \frac{C}{\sqrt{n}}.
  \end{aligned}
\end{equation*}
This finishes the proof of the Kolmogorov bound in \cref{thm:ERGM} Dobrushin's uniqueness region.

\medskip

\noindent\textbf{Step 7: Bounding $\delta_1, \delta_2$, $\delta_3$ and $\delta_1'$ in the subcritical region.}
To extend our results from Dobrushin's uniqueness region to the whole subcritical region, we
need bounds for the left-hand sides of \cref{eq:del1.3,eq:1007} in the whole subcritical region as follows. 
\begin{lemma}\label{lemma:1004}
    For any fixed graph $H$, which can have isolated vertices, with $e=|\mathcal{E}(H)|\geq 1$ and $v=|\mathcal{V}(H)|$ being its number of edges and vertices, respectively, in the subcritical region, we have
\begin{equation}
   \label{eq:1011}
  \begin{aligned}
    \mathbb{E}\left||\text{Hom}(H,G_Y,s)|- mean\right|^3 = O(n^{3v-\frac{15}{2}}). 
  \end{aligned}
\end{equation}
\end{lemma}

\begin{lemma}\label{lemma:1002}
For any fixed graph $H$, which can have isolated vertices, with $e=|\mathcal{E}(H)|\geq 1$ and $v=|\mathcal{V}(H)|\geq 3$, in the subcritical region, we have 
\begin{equation}
   \label{eq:10070}
  \begin{aligned}
    \Var\left(|\text{Hom}(H,G_Y)|- 2 n^{v-2} e p^{e-1} \sigma_{n}f(Y)\right) = O(n^{2v-\frac{5}{2}}). 
  \end{aligned}
\end{equation}
\end{lemma}
Note that the bound \cref{eq:10070} is not as good as that in \cref{eq:1007}. Although it is enough to prove the CLT, the resulting error bound \cref{eq:thmsub} is worse than \cref{eq:thmdob}. We leave a more careful study of higher-order concentration inequalities in the future work.

We first bound $\delta_1, \delta_2$, $\delta_3$ and $\delta_1'$ using \cref{lemma:1002,lemma:1004} and then prove the two lemmas. First, it is easy to see that the bound \cref{0del2} for $\delta_2$ does not change in the subcritical region. 
Moreover, the arguments leading to the bound $\delta_1\leq C/\sqrt{n}$ in \cref{eq:del1.1}--\cref{eq:del1.3} are still valid given the new \cref{lemma:1004}.
Therefore, we still have
\ben{\label{0delt1}
\delta_1\leq \frac{C}{\sqrt{n}}.
}

Next, using \cref{lemma:1002}
instead of \cref{eq:1007} in bounding $\delta_3$,
we obtain
\ben{\label{0delt3}
\delta_3=\sqrt{\Var(\Delta_3(Y))}= O(n^{-1/4}).
}

For $\delta_1'$, similar to $\delta_1$, applying \cref{lemma:1004} and H\"older's inequaliy, we obtain
\begin{equation}\label{0delt1'}
    \delta_1' \leq \frac{C}{\sqrt{n}}.
\end{equation}
Combining \cref{0delt1,0del2,0delt3,0delt1'}, we complete the proof of \cref{eq:thmsub} in \cref{thm:ERGM} for the subcritical region. 
\hfill $\square$

\medskip

It remains to prove \cref{lemma:1004,lemma:1002}. The proof extends the approach of \cite{chatterjee2007stein} and \cite{ganguly2019sub} to higher-order concentration inequalities. 

Recall we identify a graph $G$ on $n$ vertices with its edge indicators $y=\{y_e\}_{e\in \mathcal{I}}$, where $\mathcal{I}:=\{(i,j): 1\leq i<j\leq n\}$.

First, we define the grand coupling $\left(\left\{Z^y(t)\right\}_{y\in\mathcal{G}_n}\right)_{t\ge0}$ for the Glauber dynamics starting from all initial states $y$ in $\mathcal{G}_n$. They are discrete-time reversible Markov chains with ERGM \cref{ERGM} as the stationary distribution. See, for example, \cite{ganguly2019sub} for more details. For convenience, we define
\begin{align*}
\Psi(u):=\frac{e^{u}}{1+e^{u}}.
\end{align*}
Let $I$ be a uniform random edge over $\mathcal{I}$ and $U$ be a uniform random variable on $[0,1]$ independent of $I$.
For each $y\in \mathcal{G}_n$, define
\begin{equation}\label{eq:4003}    S^y:=\begin{cases}
  1& \text{ if } 0\le U\le \Psi(\Delta_IT(y)) \\
  0& \text{ if } \Psi(\Delta_IT(y))< U\le 1
\end{cases},
\end{equation}
where $T(y)= n^2 \sum_{i=1}^k \beta_i t(H_i, G_y)$ (the exponent in the definition of ERGM \cref{ERGM}) and $\Delta_I T(\cdot)$ as defined in \cref{deltah}, i.e., $\Delta_I T(y)=\sum_{i=1}^k \beta_i |\text{Hom}(H_i, G_y, I)| n^{2-|\mathcal{V}(H_j)|}$.
Given $Z^y(0)=y$, we define the next state of the Glauber dynamics $Z^y(1)$ as
\begin{align*}         \left(Z^y(1)\right)_e:=\begin{cases}
  y_e& \text{ if } e\ne I \\
  S^y& \text{ if } e=I
\end{cases},
\end{align*}
where $\left(Z^y(1)\right)_e$ and $y_e$, $e\in \mathcal{I}$, are the edge indicators of $Z^y(1)$ and $y=Z^y(0)$, respectively. Then, the next states $\left\{Z^y(1)\right\}_{y\in\mathcal{G}_n}$ are determined by the common random variables $U$ and $I$. We define $Z^y(t), t\geq 2$ similarly by choosing $I$ and $U$ independently in each step.
We also define a natural partial ordering on $\mathcal{G}_n$. We say $x\le y$ if and only if $x_e\le y_e$ for every edge $e$. 
Then we find that the grand coupling $\left(\left\{Z^y(t)\right\}_{y\in\mathcal{G}_n}\right)_{t\ge0}$ is monotone in $y$ by the definition \cref{eq:4003} and monotonicity of $|\text{Hom}(\cdot)|$. 

We first prove \cref{lemma:1002}.

\begin{proof}[Proof of \cref{lemma:1002}]
Fix any edge $s$ in $\mathcal{I}$ and recall $|\text{Hom}(H,G_Y, s)|$ denotes the number of homomorphisms of $H$ into $G_Y$ but requiring that an edge of $H$ must be mapped to the edge $s$ (no matter the edge $s$ is present in $G_Y$ or not). By symmetry, $\mathbb{E}[|\text{Hom}(H,G_Y,s)|]$ does not depend on $s$.

From \cref{eq:GN19} and \cref{eq:LLN}, we have
\be{
\left| \E |\text{Hom}(H,G_Y,s)| - 2n^{v-2}e p^{e-1}\right|\leq C n^{v-2.5}.
}
Together with \cref{eq:gangulyvar}, it suffices to prove
\ben{\label{eq:varlem5.3}
\Var\left(|\text{Hom}(H,G_Y)|-\mathbb{E}[|\text{Hom}(H,G_Y,s)|]\sigma_nf(Y)\right)=O(n^{2v-2.5}).
}
We denote
\begin{align}\label{eq:hy}
    h(y):=|\text{Hom}(H,G_y)|-\mathbb{E}[|\text{Hom}(H,G_Y,s)|]\sigma_nf(y)-\E |\text{Hom}(H,G_Y)|.
\end{align}

  We follow the approach of \cite{chatterjee2007stein} and \cite{ganguly2019sub} to obtain \cref{eq:varlem5.3}. 
  Let $Z(t), t=0,1,2,\dots$ be the Glauber dynamics starting from the steady state $Z(0)=Y$. Let $Y'=Z(1)$. From reversibility,
  $(Y,Y')$ is an exchangeable pair of random variables, i.e., $\mathcal{L}(Y, Y')=\mathcal{L}(Y', Y)$. 
  We define  
  \ben{\label{eq:Hxy}
  H(x,y)=\sum_{t=0}^{\infty}(\mathscr{P}^t h(x)-\mathscr{P}^t h(y)),
  }
  where $\mathscr{P}h(y)=\mathbb{E}[h(Y')|Y=y]$ is the Markov kernel. From Lemma 4.1 in \cite{chatterjee2005concentration}, $H(x,y)$ is an antisymmetric function and satisfies
  \begin{align*}
      h(Y)=\mathbb{E}[H(Y,Y')|Y].
  \end{align*}
  Then we have
\begin{equation}\label{eq:0034}
\begin{aligned}   &\Var\left(|\text{Hom}(H,G_Y)|-\mathbb{E}[|\text{Hom}(H,G_Y,s)|]\sigma_n f(Y)\right)=\mathbb{E}\left(h(Y)^2\right)\\
=&\mathbb{E}\left(h(Y)H(Y,Y')\right)=-\mathbb{E}\left(h(Y')H(Y,Y')\right) =\frac{1}{2}\mathbb{E}\left((h(Y)-h(Y'))H(Y,Y')\right)\\
\le&\frac12\mathbb{E}\left\{\mathbb{E}[|h(Y)-h(Y')||H(Y,Y')|\mid Y]\right\}
\end{aligned}
\end{equation}
By the construction of $(Y,Y')$, we have
\ben{\label{eq:hH}
    \mathbb{E}[|h(Y)-h(Y')||H(Y,Y')|\mid Y=y]\le\frac1N\sum_{l\in \mathcal{I}}\left | h(y^{(l,1)})-h(y^{(l,0))}) \right | \left | H(y^{(l,1)},y^{(l,0)}) \right | 
}
where $N={n\choose 2}$ is the number of edges in the complete graph $\mathcal{I}$ with $n$ vertices
and the definitions of $y^{(l,1)}$ and $y^{(l,0)}$ are the same as $x^{(l,1)}$ and $x^{(l,0)}$ above \cref{deltah}. 
From \cref{eq:Hxy}, we bound  $H(y^{(l,1)},y^{(l,0)})$ as
\begin{equation}\label{eq:Hbound}
\left|H(y^{(l,1)}),y^{(l,0))})\right|\le\sum_{t=0}^{\infty}\left|\mathbb{E}h(y^{(l,1)}(t))-\mathbb{E}h(y^{(l,0)}(t))\right|
\end{equation}
where $y^{(l,1)}(t)$ and $y^{(l,0)}(t)$ are the grand coupling $\left(Z^y(t)\right)_{t\ge0}$ from initial configurations $y^{(l,1)}$ and $y^{(l,0)}$, respectively. 
From the definition of $h(\cdot)$ in \cref{eq:hy} and accouting for the remaining $v-2$ vertices of $H$ after fixing an edge $e$, we have
\ben{\label{eq:diffh}
    \left|h(y^{(l,1))}(t))-h(y^{(l,0)}(t))\right|=O(n^{v-2})\sum_{e\in \mathcal{I}}\mathbbm{1}\left\{(y^{(l,1)}(t))_e\ne(y^{(l,0)}(t))_e\right\}.
}
Let $r(y,l,t)$ be an $N={n\choose 2}$ dimensional vector with $$r(y,l,t)_e:=\mathbb{P}((y^{(l,1)}(t))_e\ne (y^{(l,0)}(t))_e).$$ 
From the second displayed equation (line 8) on page 2283 in \cite{ganguly2019sub} with $v_l=1$ for $1\leq l\leq N$, we have 
\begin{equation}
    \label{rbound}
    \sum_{t=0}^{\infty}\|r(y,l,t)\|_1=O(n^2),
\end{equation}
where 
$\|r(y,l,t)\|_1=\mathbb{E}\left[\sum_{e\in \mathcal{I}}\mathbbm{1}\left\{(y^{l+}(t))_e\ne(y^{l-}(t))_e\right\}\right]$.
Using \cref{rbound,eq:Hbound,eq:diffh}, we obtain
\ben{\label{eq:Hbound2}
\left|H(y^{(l,1)},y^{(l,0)})\right|=\sum_{t=0}^{\infty}O(n^{v-2})\|r(y,l,t)\|_1=O(n^{v}).
}
From \cref{eq:0034,eq:hH,eq:Hbound2}, we obtain 
\besn{\label{eq:varhy}
    &\mathrm{Var}(h(Y))\le\mathbb{E}\left[\frac1N\sum_{l\in \mathcal{I}}\left | h(Y^{(l,1)})-h(Y^{(l,0)}) \right | \left | H(Y^{(l,1)},Y^{(l,0)}) \right | \right]\\
    =&\frac1N\sum_{l\in \mathcal{I}} O(n^v)E\left| h(Y^{(l,1)})-h(Y^{(l,0)})\right|\\
    =&\frac1N\sum_{l\in \mathcal{I}} O(n^v)\sqrt{\mathbb{E}(|\text{Hom}(H,G_Y,l)|-\mathbb{E}|\text{Hom}(H,G_Y,l)|)^2},
}
where we used $h(Y^{(l,1)})-h(Y^{(l,0)})=|\text{Hom}(H,G_Y,l)|-\mathbb{E}|\text{Hom}(H,G_Y,s)|$.

Next, we follow a similar approach to estimate $\mathrm{Var}(|\text{Hom}(H,G_{Y},l)|)$. Let $g_l(Y)=|\text{Hom}(H,Y,l)|$ and $G_l(x,y)=\sum_{t=0}^{\infty}(\mathscr{P}^t g_l(x)-\mathscr{P}^t g_l(y))$. Similar to \cref{eq:0034} and \cref{eq:hH}, we have
\begin{equation}\label{eq:0038}
\begin{aligned}
    &\mathbb{E}(|\text{Hom}(H,G_{Y},l)|-\mathbb{E}|\text{Hom}(H,G_{Y},s)|)^2
    =\frac12\mathbb{E}[(g_l(Y)-g_l(Y'))G_l(Y,Y')]\\
\end{aligned}
\end{equation}
and
\besn{\label{eq:00039}
   &\mathbb{E}[(g_l(Y)-g_l(Y'))G_l(Y,Y')|Y=y]\\
    \le&\frac1N\sum_{r\in \mathcal{I}: r\ne l} \left[|g_l(y^{(r,1)})-g_l(y^{(r,0)})||G_l(y^{(r,1)},y^{(r,0)})|\right].
}
Similar to the estimation of $H(y^{(l,1)},y^{(l,0)})$, we can bound $G_l(y^{(r,1)},y^{(r,0)})$ using \cref{rbound} as
$$
\begin{aligned}  &|G_l(y^{(r,1)},y^{(r,0)})|\le\sum_{t=0}^{\infty}|\mathbb{E}g_l(y^{(r,1)}(t))-\mathbb{E}g_l(y^{(r,0)}(t))|\\
=&\sum_{t=0}^{\infty}O(n^{v-3})\|r(y,r,t)\|_1
    =O(n^{v-1}).\\
\end{aligned}
$$
Furthermore, for any $y\in \mathcal{G}_n$, 
$$
\begin{aligned}
   |g_l(y^{(r,1)})-g_l(y^{(r,0)})|=|\text{Hom}(H,G_y,l,r)|,
\end{aligned}
$$
where $|\text{Hom}(H,G_y,l,r)|$ denotes the number of homomorphisms of $H$ into $G_y$ but requiring to cover the edges $l$ and $r$ (no matter the edges $l$ and $r$ are present in $G_y$ or not).
When the edges $l$ and $r$ share a same vertex, we have $|\text{Hom}(H,G_y,l,r)|= O(n^{v-3})$, and there are $O(n)$ number of pair $(l,r)$ for fixed edge $l$. When the edges $l$ and $r$ do not share a same vertex, we have $ |\text{Hom}(H,G_y,l,r)|= O(n^{v-4})$, and there are $O(n^2)$ number of pairs $(l,r)$ for a fixed edge $l$. Therefore, we have
\begin{equation}\label{eq:0039}
 \begin{aligned}
    &\quad\mathbb{E}[|(g_l(Y)-g_l(Y'))G_l(Y,Y')||Y]
    \\&= \frac{O(n^{v-1})}{N}\left(\sum_{r:\  r\ \mbox{and}\ l\ \mbox{are
connected}}|\text{Hom}(H,G_{Y},l,r)|\right.\\    &\left.\quad+\sum_{r:\  r\ \mbox{and}\ l\ \mbox{are disconnected}}|\text{Hom}(H,G_{Y},l,r)|\right)\\
    &= O(n^{v-3})(O(n)O(n^{v-3})+O(n^2)O(n^{v-4}))=O(n^{2v-5}).
\end{aligned}   
\end{equation}
Going back to \cref{eq:0038}, we have
\ben{\label{eq:varhoml}
\begin{aligned}
    &\mathrm{Var}(|\text{Hom}(H,G_Y,l)|)=\frac12\E\left[\mathbb{E}[(g_l(Y)-g_l(Y'))G_l(Y,Y')|Y] \right]=O(n^{2v-5}).
\end{aligned}
}
Combining \cref{eq:varhy} with the above estimate, we finally obtain
\begin{equation*}
 \begin{aligned}
    \mathrm{Var}(h(Y))= O(n^v)O(n^{v-\frac52})=O(n^{2v-\frac52}).
\end{aligned}   
\end{equation*}
This proves \cref{eq:varlem5.3}.
\end{proof}

Finally, we prove \cref{lemma:1004}.

\begin{proof}[Proof of \cref{lemma:1004}]
The proof is similar to proving \cref{eq:varhoml}.
    Let \be{g_s(Y)=|\text{Hom}(H,G_Y,s)|-mean.} Then $\mathbb{E}g_s(Y)=0$ and $g_s(Y)=\mathbb{E}[G_s(Y,Y')\mid Y]$, where $G_s(x,y)=\sum_{t=0}^{\infty}(\mathscr{P}^t g_s(x)-\mathscr{P}^t g_s(y))$, $\mathscr{P}g_s(y)=\mathbb{E}[g_s(Y')|Y=y]$. We have, from the exchangeability of $(Y, Y')$ and antisymmetry of $G_s(x, y)$,
$$
\begin{aligned}
    &\mathbb{E}|g_s(Y)|^3=\mathbb{E}|g_s(Y)|g_s(Y)^2=\mathbb{E}|g_s(Y)|g_s(Y)G_s(Y,Y')\\
    =&\mathbb{E}|g_s(Y')|g_s(Y')G_s(Y',Y)    =-\mathbb{E}|g_s(Y')|g_s(Y')G_s(Y,Y')\\
    =&\frac12\mathbb{E}\left\{(|g_s(Y)|g_s(Y)-|g_s(Y')|g_s(Y'))G_s(Y,Y')\right\}\\
    \le&\frac12\mathbb{E}\left\{|g_s(Y)-g_s(Y')||g_s(Y)||G_s(Y,Y')|\right\}+\frac12\mathbb{E}\left\{|g_s(Y)-g_s(Y')||g_s(Y')||G_s(Y,Y')|\right\}\\
    =&\mathbb{E}\left\{|g_s(Y)-g_s(Y')||g_s(Y)||G_s(Y,Y')|\right\}\\
    =&\mathbb{E}\left\{|g_s(Y)|\mathbb{E}[|g_s(Y)-g_s(Y')||G_s(Y,Y')|\mid Y]\right\}
\end{aligned}
$$
From \cref{eq:0039}, we know $\mathbb{E}[|g_s(Y)-g_s(Y')||G_s(Y,Y')|\mid Y]=O(n^{2v-5})$. Therefore,
$$
\mathbb{E}|g_s(Y)|^3= O(n^{2v-5})\mathbb{E}|g_s(Y)|= O(n^{2v-5})\sqrt{\mathrm{Var}(g_s(Y))}= O(n^{\frac32(2v-5)}),
$$
where we used \cref{eq:varhoml} in the last step.
\end{proof}

\color{black}
\section{Proof of Example \ref{ex:CW}}\label{sec:CW}

We use the same notation as in \cref{theorem:1} and \cref{ex:CW}. We use $C$ to denote positive constants depending only on $\beta$. It is straightforward to compute that
\begin{equation}
   \label{eq:2001}
  \begin{aligned}
  g(X)-g(X^{(i)})= \frac{\beta}{2N} (2 s(X) (X_{i}-X_{i}') - (X_{i}-X_{i}')^{2})
  \end{aligned}
\end{equation}
and 
\begin{equation}
   \label{eq:2002}
  \begin{aligned}
    f(X)-f(X^{(i)})= f(X^{[i]})-f(X^{[i-1]}) = \frac{X_{i}-X_{i}'}{\sigma_{N}}.
  \end{aligned}
\end{equation}
For any integers $0\leq u, v\leq 3$, from the facts that $|X_{i}|=1$, $\frac{1}{a}\E [h(X)\lvert f(X) \rvert^2]=\E [|f(Y)|^2]\leq C$ (see \cref{eq:transfer} for the equation) and $|g(X)-g(X^{(i)})|\leq C$, we have 
\begin{equation}
   \label{eq:2003}
  \begin{aligned}
  &\frac{1}{a \sigma_{N}^{v}}\E\left\{ h(X)  |g(X)-g(X^{(i)})|^{u} |X_{i}-X_{i}'|^{v} \right\}\\
  \leq & \frac{C}{a \sigma_{N}^{v}}\E\left\{ h(X)  |g(X)-g(X^{(i)})|^{u}\right\} \\
  \leq & \frac{C}{a \sigma_{N}^{v} N^{u-1_{\{ u=3 \}}}}  \E \left\{ h(X) (\lvert s(X) \rvert^{u-1_{\{ u=3 \}}} + 1) \right\}\\
  \leq & \frac{C}{a \sigma_{N}^{v-(u-1_{\{ u=3 \}})} N^{u-1_{\{ u=3 \}}}}  \E \left\{h(X) (\lvert f(X) \rvert^{u-1_{\{ u=3 \}}} + 1)\right\} \\
  \leq & \frac{C}{\sigma_{N}^{v-(u-1_{\{ u=3 \}})} N^{u-1_{\{ u=3 \}}}}. 
  \end{aligned}
\end{equation}
Combining \cref{eq:2003} and the fact that $|g(X)-g(X^{(i)})|\leq C$,
we obtain
\begin{equation}
   \label{eq:2004}
  \begin{aligned}
  \delta_{1}\leq \frac{C}{\sqrt{N}}.
  \end{aligned}
\end{equation}
From \cref{eq:2001}, \cref{eq:2002} and $\lvert X_{i} \rvert=1$, we have
\begin{equation}
   \label{eq:2005}
  \begin{aligned}
    \Delta_{1,i}(X)= \frac{1}{2 \sigma_{N}^{2}} \E \left[ (X_{i}-X_{i}')^{2}|X \right] = \frac{1}{2 \sigma_{N}^{2}} \E \left[ 2-2X_{i}X_{i}'|X \right]= \frac{1}{\sigma_{N}^{2}},
  \end{aligned}
\end{equation}
and 
\begin{equation}
   \label{eq:2006}
  \begin{aligned}
    \Delta_{2,i}(X)=& \frac{\beta}{4 N \sigma_{N}}\E \left[ (X_{i}-X_{i}')(2 s(X) (X_{i}-X_{i}') - (X_{i}-X_{i}')^{2})|X \right]\\
                =& \frac{\beta}{ N \sigma_{N}}(s(X)-X_{i}).
  \end{aligned}
\end{equation}
Therefore, 
\begin{equation}
   \label{eq:2007}
  \begin{aligned}
  b=\frac{N}{\sigma_{N}^{2}}= 1-\beta.
  \end{aligned}
\end{equation}
From \cref{eq:2005,eq:2006,eq:2007}, we have
\begin{equation}
   \label{eq:2008}
  \begin{aligned}
  \delta_{2}= 0
  \end{aligned}
\end{equation}
and 
\begin{equation}
   \label{eq:2009}
  \begin{aligned}
    \delta_{3} =\sqrt{ \Var\left(\frac{\beta}{\sigma_{N}} s(Y) - \frac{\beta}{N \sigma_{N}} s(Y) - \beta f(Y)\right)}\leq \frac{C}{N}
  \end{aligned}.
\end{equation}
We remark that we can derive the expression of $\sigma_N^2$ from \cref{eq:2007} and \cref{eq:2009} even if we did not know it in the first place. This idea is used to derive the new expression of the asymptotic variance \cref{eq:sigmasq} for the ERGM. 

Using \cref{theorem:1} with the above bounds, we obtain the Wasserstein bound 
\begin{equation}
   \label{eq:2010}
  \begin{aligned}
  d_{\text{Wass}}(W,Z)\leq \frac{C}{\sqrt{N}}.
  \end{aligned}
\end{equation}
For the Kolmogorov bound, we choose $D^{*}_{i}(X, X')= {2}/{\sigma_{N}}$. Then, for the second term in the expression of $\delta_{1}'$,
\begin{equation}
   \label{eq:2012}
  \begin{aligned}
  &\frac{1}{a}\E\left\{ h(X) \Big\lvert\sum_{ i=1 }^{ N } \E \big[ D^*_{i}(X,X')(f(X)- f(X^{(i)}))|X\big]\Big\rvert\right\}\\
=&\frac{2}{a \sigma_{N}} \E\left\{ h(X) \Big\lvert\sum_{ i=1 }^{ N } \E \big[ (f(X)- f(X^{(i)}))|X\big]\Big\rvert\right\}\\
=&\frac{2}{\sigma_{N}} \E \lvert f(Y) \rvert \leq \frac{C}{\sqrt{N}}. \\
  \end{aligned}
\end{equation}
Similar to \cref{eq:2004}, using \cref{eq:2003}, the other terms in the expression of $\delta_{1}'$ are also of the order $O(1/\sqrt{N})$. Thus, 
\begin{equation}
   \label{eq:2013}
  \begin{aligned}
    \delta_{1}'\leq \frac{C}{\sqrt{N}},
  \end{aligned}
\end{equation}
and we obtain 
\begin{equation}
   \label{eq:2014}
  \begin{aligned}
  d_{\text{Kol}}(W,Z)\leq \frac{C}{\sqrt{N}}.
  \end{aligned}
\end{equation}

\section*{Acknowledgements}

Fang X. was partially supported by Hong Kong RGC GRF 14305821, 14304822, 14303423 and a CUHK direct grant.
Liu S.H. was partially supported by National Nature Science Foundation of China NSFC 12301182.
Shao Q.M. was partially supported by National Nature Science Foundation of China NSFC 12031005 and Shenzhen Outstanding Talents Training Fund, China.

\bibliographystyle{apalike}
\bibliography{reference}

\begin{thebibliography}{}

\bibitem[Bhamidi et~al., 2011]{bhamidi2011mixing}
Bhamidi, S., Bresler, G., and Sly, A. (2011).
\newblock Mixing time of exponential random graphs.
\newblock {\em The Annals of Applied Probability}, 21(6):2146--2170.

\bibitem[Bianchi et~al., 2024]{bianchi2024limit}
Bianchi, A., Collet, F., and Magnanini, E. (2024).
\newblock Limit theorems for exponential random graphs.
\newblock {\em Ann. Appl. Probab.}, 34(5):4863--4898.

\bibitem[Chatterjee, 2005]{chatterjee2005concentration}
Chatterjee, S. (2005).
\newblock Concentration inequalities with exchangeable pairs (ph. d. thesis).
\newblock {\em arXiv preprint math/0507526}.

\bibitem[Chatterjee, 2007]{chatterjee2007stein}
Chatterjee, S. (2007).
\newblock Stein's method for concentration inequalities.
\newblock {\em Probability Theory and Related Fields}, 138(1-2):305.

\bibitem[Chatterjee, 2008]{chatterjee2008new}
Chatterjee, S. (2008).
\newblock A new method of normal approximation.
\newblock {\em The Annals of Probability}, 36(4):1584--1610.

\bibitem[Chatterjee and Diaconis, 2013]{chatterjee2013estimating}
Chatterjee, S. and Diaconis, P. (2013).
\newblock Estimating and understanding exponential random graph models.
\newblock {\em The Annals of Statistics}, 41(5):2428--2461.

\bibitem[Chatterjee and Shao, 2011]{chatterjee2011nonnormal}
Chatterjee, S. and Shao, Q.-M. (2011).
\newblock Nonnormal approximation by stein’s method of exchangeable pairs
  with application to the {C}urie--{W}eiss model.
\newblock {\em The Annals of Applied Probability}, 21(2):464--483.

\bibitem[Chen et~al., 2013]{chen2013stein}
Chen, L. H.~Y., Fang, X., and Shao, Q.-M. (2013).
\newblock From stein identities to moderate deviations.
\newblock {\em The Annals of Probability}, 41(1):262--293.

\bibitem[Chen et~al., 2010]{chen2010normal}
Chen, L. H.~Y., Goldstein, L., and Shao, Q.-M. (2010).
\newblock {\em Normal Approximation by Stein’s Method}.
\newblock Springer Science \& Business Media.

\bibitem[Chen and R{\"o}llin, 2010]{chen2010stein}
Chen, L. H.~Y. and R{\"o}llin, A. (2010).
\newblock Stein couplings for normal approximation.
\newblock {\em arXiv preprint arXiv:1003.6039}.

\bibitem[Ding and Fang, 2024]{ding2024second}
Ding, W.-Y. and Fang, X. (2024).
\newblock Second-order approximation of exponential random graph models.
\newblock {\em arXiv preprint arXiv:2401.01467}.

\bibitem[Ellis and Newman, 1978a]{ellis1978limit}
Ellis, R.~S. and Newman, C.~M. (1978a).
\newblock Limit theorems for sums of dependent random variables occurring in
  statistical mechanics.
\newblock {\em Zeitschrift f{\"u}r Wahrscheinlichkeitstheorie und verwandte
  Gebiete}, 44(2):117--139.

\bibitem[Ellis and Newman, 1978b]{ellis1978statistics}
Ellis, R.~S. and Newman, C.~M. (1978b).
\newblock The statistics of {C}urie--{W}eiss models.
\newblock {\em Journal of Statistical Physics}, 19(2):149--161.

\bibitem[Fang and Koike, 2021]{fang2021high}
Fang, X. and Koike, Y. (2021).
\newblock High-dimensional central limit theorems by {S}tein’s method.
\newblock {\em The Annals of Applied Probability}, 31(4):1660--1686.

\bibitem[Fang and Koike, 2022]{fang2022new}
Fang, X. and Koike, Y. (2022).
\newblock New error bounds in multivariate normal approximations via
  exchangeable pairs with applications to {W}ishart matrices and fourth moment
  theorems.
\newblock {\em The Annals of Applied Probability}, 32(1):602--631.

\bibitem[Frank and Strauss, 1986]{frank1986markov}
Frank, O. and Strauss, D. (1986).
\newblock Markov graphs.
\newblock {\em Journal of the American Statistical Association},
  81(395):832--842.

\bibitem[Ganguly and Nam, 2024]{ganguly2019sub}
Ganguly, S. and Nam, K. (2024).
\newblock Sub-critical exponential random graphs: concentration of measure and
  some applications.
\newblock {\em Transactions of the American Mathematical Society},
  377(04):2261--2296.

\bibitem[G{\"o}tze et~al., 2019]{gotze2019higher}
G{\"o}tze, F., Sambale, H., and Sinulis, A. (2019).
\newblock Higher order concentration for functions of weakly dependent random
  variables.
\newblock {\em Electronic Journal of Probability}, 24:1--19.

\bibitem[Hoeffding, 1948]{hoeffding1948}
Hoeffding, W. (1948).
\newblock A class of statistics with asymptotically normal distribution.
\newblock {\em The Annals of Mathematical Statistics}, 19(3):293--325.

\bibitem[Holland and Leinhardt, 1981]{holland1981exponential}
Holland, P.~W. and Leinhardt, S. (1981).
\newblock An exponential family of probability distributions for directed
  graphs.
\newblock {\em Journal of the American Statistical Association},
  76(373):33--50.

\bibitem[Mukherjee and Xu, 2023]{mukherjee2023statistics}
Mukherjee, S. and Xu, Y. (2023).
\newblock Statistics of the two star {ERGM}.
\newblock {\em Bernoulli}, 29(1):24--51.

\bibitem[Park and Newman, 2004]{park2004statistical}
Park, J. and Newman, M.~E. (2004).
\newblock Statistical mechanics of networks.
\newblock {\em Physical Review E}, 70(6):066117.

\bibitem[Reinert and Ross, 2019]{reinert2019approximating}
Reinert, G. and Ross, N. (2019).
\newblock Approximating stationary distributions of fast mixing {G}lauber
  dynamics, with applications to exponential random graphs.
\newblock {\em The Annals of Applied Probability}, 29(5):3201--3229.

\bibitem[Sambale and Sinulis, 2020]{sambale2020logarithmic}
Sambale, H. and Sinulis, A. (2020).
\newblock Logarithmic {S}obolev inequalities for finite spin systems and
  applications.
\newblock {\em Bernoulli}, 26(3):1863--1890.

\bibitem[Shao and Zhang, 2019]{shaozhang2019}
Shao, Q.-M. and Zhang, Z.-S. (2019).
\newblock {Berry–Esseen bounds of normal and nonnormal approximation for
  unbounded exchangeable pairs}.
\newblock {\em The Annals of Probability}, 47(1):61--108.

\bibitem[Shao and Zhang, 2024]{shao2024berry}
Shao, Q.-M. and Zhang, Z.-S. (2024).
\newblock Berry--{E}sseen bounds for functionals of independent random
  variables.
\newblock {\em Manuscript}.

\bibitem[Stein, 1972]{stein1972bound}
Stein, C. (1972).
\newblock A bound for the error in the normal approximation to the distribution
  of a sum of dependent random variables.
\newblock In {\em Proceedings of the Sixth Berkeley Symposium on Mathematical
  Statistics and Probability, Volume 2: Probability Theory}, volume~6, pages
  583--603. University of California Press.

\bibitem[Wasserman and Faust, 1994]{wasserman1994social}
Wasserman, S. and Faust, K. (1994).
\newblock {\em Social Network Analysis: Methods and Applications}.
\newblock Cambridge University Press.

\end{thebibliography}

\end{document}